\documentclass[12pt]{amsart}

\usepackage{latexsym}
\usepackage{a4} 
\usepackage{xypic}
\usepackage{amssymb}
\usepackage{amssymb, amsmath}
\usepackage{graphicx}
\usepackage{tikz}
\usepackage{hyperref}
\usepackage{amsthm}
\usepackage{setspace}

\newtheorem{Thm}{Theorem}[section] 
 
\newtheorem{Lem}[Thm]{Lemma} 
\newtheorem{Cor}[Thm]{Corollary} 
\theoremstyle{definition}
\newtheorem{Rem}[Thm]{Remark} 
 
\theoremstyle{definition}
\newtheorem{Def}[Thm]{Definition}
\numberwithin{equation}{section}

\newcommand{\Con}{\mathrm{Con}}

\newcommand{\Clg}{\mathrm{Clg}}
\newcommand{\Cid}{\mathrm{Cig}}

\newcommand{\TD}{\mathrm{TD}}
\newcommand{\tD}{\mathrm{tD}}
\newcommand{\DEGS}{\mathrm{DEGS}}
\newcommand{\MInd}{\mathrm{MInd}}
\newcommand{\MON}{\mathrm{MON}}
\newcommand{\Var}{\mathrm{Var}}
\newcommand{\id}{\mathrm{Id}}

\newcommand{\vv}[1]{\mathcal {#1}}

\newcommand{\Pol}{\mathrm{Pol}}
\newcommand{\Clo}{\mathrm{Clo}}

\newcommand{\A}{{\mathbf A}}
\newcommand{\B}{{\mathbf B}}

\newcommand{\FF}{{\mathbb F}}

\xymatrixrowsep {1pc}
\xymatrixcolsep {3.1pc}

\newcommand{\edge}[1]{\ar@{-}[#1]}

\newcommand{\lattf}[2]{
	\begin{xy}
		(3.5,0):(0,1.2)::
		(1,0)*+[o][F-]{ }="0";
		(0,3)*+[o][F-]{ }="A";
		(3,1)*+[o][F-]{ }="B";
		(2,4)*+[o][F-]{ }="1";
		"0";"A"**\dir{-} ?*_!/7pt/{#1};
		"0";"B"**\dir{-} ?*^!/7pt/{#2};
		"A";"1"**\dir{-} ?*_!/7pt/{#2};
		"B";"1"**\dir{-} ?*^!/7pt/{#1};
	\end{xy}
}
\newcommand{\lattl}[2]{
	\begin{xy}
		(3.5,0):(0,1.2)::
		(1,0)*+[o][F-]{ }="0";
		(0,3)*+[o][F-]{ }="A";
		(2,4)*+[o][F-]{ }="1";
		"0";"A"**\dir{-} ?*_!/7pt/{#1};
		"A";"1"**\dir{-} ?*_!/7pt/{#2};
	\end{xy}
}
\newcommand{\lattr}[2]{
	\begin{xy}
		(3.5,0):(0,1.0)::
		(1,0)*+[o][F-]{ }="0";
		(3,1)*+[o][F-]{ }="B";
		(2,4)*+[o][F-]{ }="1";
		"0";"B"**\dir{-} ?*^!/7pt/{#2};
		"B";"1"**\dir{-} ?*^!/7pt/{#1};
	\end{xy}
}
\newcommand{\latts}{
	\begin{xy}
		(3,0):(0,1.0)::
		(0,0)*+[o][F-]{ }="0";
		(0,4)*+[o][F-]{ }="1";
		"0";"1"**\dir{-} ?*^!/7pt/{3};
	\end{xy}
}

\newcommand{\ZZ}{{\mathbb{Z}}}
\newcommand{\NN}{{\mathbb{N}}}

\DeclareMathAlphabet\mathbfsl {T1}{cmr}{bx}{it}

\title[Clones on $\ZZ_{pq}$]{Clones containing the Mal'cev operation of $\ZZ_{pq}$}
\author{Stefano Fioravanti}

\address{Stefano Fioravanti,
	Institut f\"ur Algebra,
	Johannes Kepler Universit\"at Linz,
	4040 Linz,
	Austria}
\email{\tt stefano.fioravanti66@gmail.com}
\subjclass[2018]{08A40}
\urladdr{http://www.jku.at/algebra}
\thanks{Supported by the Austrian Science Fund (FWF):P29931.}
\keywords{Clonoids, Clones}
\date{\today}

\begin{document}
	
	\begin{abstract}

		We investigate finitary functions from $\mathbb{Z}_{pq}$ to $\mathbb{Z}_{pq}$ for two distinct prime numbers $p$ and $q$. We show that the lattice of all clones on the set $\mathbb{Z}_{pq}$ which contain the addition of $\mathbb{Z}_{pq}$ is finite. We provide an upper bound for the cardinality of this lattice through an injective function to the direct product of the lattice of all $(\ZZ_p,\ZZ_q)$-linearly closed clonoids to the $p+1$ power and the lattice of all $(\ZZ_q,\ZZ_p)$-linearly closed clonoids to the $q+1$ power. These lattices are studied in \cite{Fio.CSOF} and there we can find the exact cardinality of them. Furthermore, we prove that these clones can be generated by a set of functions of arity at most $max(\{p,q\})$.

	\end{abstract}

	\maketitle

\section{Introduction}

The investigation of the lattice of all clones on a set $A$ has been a fecund field of research in general algebra with results such as Emil Post's characterization of the lattice of all clones on a two-element set \cite{Pos.TTVI}. This branch was developed further, e. g., in \cite{Ros.MCOA,PK.FUR,Sze.CIUA} and starting from \cite{BJK.TCOC}, clones are used to study the complexity of certain constraint satisfaction problems (CSPs).

The aim of this paper is to describe the lattice of those clones on the set $\ZZ_{pq}$ that contain the operation of addition of $\ZZ_{pq}$,  with $p$ and $q$ distinct primes. Thus we want to study the part of the lattice of all clones on $\ZZ_{pq}$ which is above the clone of all linear mappings. 

In \cite{Idz.CCMO} P. Idziak characterized the number of polynomial Mal'cev clones (clones containing the constants and a Mal'cev term) on a finite set $A$, which is finite if and only if $|A|\leq 3$. In \cite{Bul.PCCT} A. Bulatov shows a full characterization of all infinitely many polynomial clones on the sets $\ZZ_p \times \ZZ_p$ and $\ZZ_{p^2}$ that contain $+$, where $p$ is a prime. Moreover, a description of polynomial clones on $\ZZ_{pq}$ containing the sum for distinct primes $p$ and $q$ is given in \cite{AM.PCOG} and polynomial clones containing $+$ on $\ZZ_n$, for $n$ squarefree, are described in \cite{May.PCOS}.

In \cite{Kre.CFSO} S. Kreinecker proved that there are infinitely many non-finitely generated clones above $\Clo(\langle\mathbb{Z}_p \times \mathbb{Z}_p , +\rangle)$ for a prime $p > 2$. 

In this paper we will make often use of the concept of $(\FF_p,\FF_q)$-linearly closed clonoid as defined in \cite[Definition $1.1$]{Fio.CSOF}. We recall this definition.

\begin{Def}
	\theoremstyle{definition}
	\label{DefClo}
	Let $p$ and $q$ be powers of different primes, and let $\mathbb{F}_p$ and $\mathbb{F}_q$ be two fields of orders $p$ and $q$. An \emph{$(\mathbb{F}_p,\mathbb{F}_q)$-linearly closed clonoid} is a non-empty subset $C$ of $\bigcup_{n \in \mathbb{N}} \mathbb{F}^{\mathbb{F}^n_q}_p$ with the following properties:
	
	\begin{enumerate}
		\item[(1)] for all $n \in \NN$, $f,g \in C^{[n]}$ and $a, b \in \mathbb{F}_p$:
		
		\begin{center}
			$af + bg \in C^{[n]}$;
		\end{center}
		
		\item[(2)] for all $m,n \in \NN$, $f \in C^{[m]}$ and $A \in \mathbb{F}^{m \times n}_q$:
		
		\begin{center}
			$g: (x_1,\dots,x_n) \mapsto f(A\cdot (x_1,\dots,x_n)^t)$ is in $C^{[n]}$.
		\end{center}
		
	\end{enumerate}
\end{Def}

In \cite[Thoerems $1.2$ and $1.3$]{Fio.CSOF} we can find a complete description of the lattice of all $(\FF_p,\FF_q)$-linearly closed clonoids with $p$ and $q$ powers of distinct primes.

In Section \ref{EmbClo} we show an embedding of the lattice of all $(\mathbb{Z}_p,\ZZ_q)$-linearly closed clonoids into the lattice of all clones above $\Clo(\langle\ZZ_{pq},+\rangle)$, where $p$ and $q$ are distinct primes.

\begin{Thm}
	\label{ThEmbClonoids}
	Let $p$ and $q$ be two distinct prime numbers. Then the lattice of all $(\mathbb{Z}_p,\ZZ_q)$-linearly closed clonoids is embedded in the lattice of all clones above $\Clo(\langle\ZZ_{pq},+\rangle)$.
\end{Thm}

In Section \ref{SecIndipAlg} we investigate the part of the lattice of all clones above $\Clo(\langle\ZZ_{pq},+\rangle)$ which is composed by all clones that can be split in two components over $\ZZ_p$ and over $\ZZ_q$. These are the clones which preserve both $\pi_1$ and $\pi_2$. We will show a characterization of these clones that can be proven also using \cite[Lemma $6.1$]{AM.IOAW}.

\begin{Thm}\label{ThmembeddingClones}
	
	Let $p$ and $q$ be distinct prime numbers. Then there is an isomorphism between the lattice of all clones above $\Clo(\langle\ZZ_{pq},+\rangle)$ which preserve $\{\pi_1, \pi_2\}$ and the direct product of the lattices of all clones above $\Clo(\langle\ZZ_{p},+\rangle)$ and of all clones above $\Clo(\langle\ZZ_{q},+\rangle)$.
	
\end{Thm}

In Section \ref{AGenBoun} we can find the main results of this paper that regard the cardinality of the lattice of all clones on $\ZZ_{pq}$ that contain all linear mappings. 

\begin{Thm}
	\label{Thmgeneralembedding}
	Let $p$ and $q$ be distinct prime numbers and let $\mathbf{Clo}^\mathcal{L}(\langle \ZZ_{pq},+\rangle)$ be the lattice of all clones containing $\Clo(\langle \ZZ_{pq},+\rangle)$. Then there is an injective function from $\mathbf{Clo}^{\mathcal{L}}(\langle \ZZ_{pq},+\rangle)$ to the direct product of the lattice of all $(\ZZ_p,\ZZ_q)$-linearly closed clonoids, $\mathcal{L}(\ZZ_p,\ZZ_q)$, to the $p+1$ power and the lattice of all $(\ZZ_q,\ZZ_p)$-linearly closed clonoids, $\mathcal{L}(\ZZ_q,\ZZ_p)$, to the $q+1$ power, i. e:
	
	\begin{center}

		$\Clo^{\mathcal{L}}(\langle \ZZ_{pq},+\rangle)\hookrightarrow \mathcal{L}(\ZZ_p,\ZZ_q)^{p+1} \times \mathcal{L}(\ZZ_q,\ZZ_p)^{q+1}.$
		
	\end{center}

\end{Thm}

Furthermore, from this result we can obtain a bound for the number of clones on $\ZZ_{pq}$ that contain $\Clo(\langle \ZZ_{pq},+\rangle)$. 

\begin{Cor}
	\label{Corfinale}
	Let $p$ and $q$ be distinct prime numbers. Let $\prod_{i =1}^n p_i^{k_i}$ and $\prod_{i =1}^s r_i^{d_i}$ be the factorizations of respectively $g_p= x^{q-1} -1$ in $\mathbb{Z}_p[x]$ and $g_q = x^{p-1} -1$ in $\mathbb{Z}_q[x]$ into their irreducible divisors. Then the number $k$ of clones containing $\Clo(\langle\ZZ_{pq},+\rangle)$ is bounded by:
	
	\begin{equation}
	2(\prod_{i =1}^n(k_i +1) + \prod_{i =1}^s(d_i +1) ) - 1\leq k \leq 2^{p+q+2}\prod_{i =1}^n(k_i +1)^{p+1}\prod_{i =1}^s(d_i +1)^{q+1} \leq 2^{qp+q+p}.
	\end{equation}
\end{Cor}

This theorem extends the finiteness result in \cite{AM.PCOG} to clones on $\ZZ_{pq}$ which do not necessarily contain constants, with $p$ and $q$ distinct primes. The main ingredient we used is \cite[Theorem $1.3$]{Fio.CSOF}, which provides a characterization of the lattice of all $(\FF_p,\FF_q)$-linearly closed clonoids for all $p$ and $q$ powers of distinct primes.

We can also use the proof of Theorem \ref{Thmgeneralembedding} to find a concrete bound on the arity of the generators of clones containing $\Clo(\langle\ZZ_{pq},+\rangle)$.

\begin{Cor}
	\label{CorArFun}
	Let $p$ and $q$ be distinct prime numbers. Then the clones containing $\Clo(\langle\ZZ_{pq},+\rangle)$ can be generated by a set of functions of arity at most $max(\{p,q\})$.
	
\end{Cor}

This corollary provides a bound for the arity of the generators of a clone containing $\Clo(\langle\ZZ_{pq},+\rangle)$ which is  $max(\{p,q\})$ and gives a rather low and unexpected bound for the arity of the functions that really determine the clones.

Furthermore, we use this result to find a description of some parts of the lattice of all clones above $\Clo(\langle\ZZ_{pq},+\rangle)$. We denote by $\pi_1, \pi_2$ the kernels of the two binary projections from $\ZZ_p \times \ZZ_q$ to respectively $\ZZ_p$ and $\ZZ_q$. 

\begin{Thm}\label{Thmembaddingg}
	Let $p$ and $q$ be distinct prime numbers. Then there is an injective function from the lattice of all clones above $\Clo(\langle \ZZ_{pq},+\rangle)$ that preserve $\pi_1$ and $[\pi_1,\pi_1] = 0$ to the direct product of the lattice of all clones above $\Clo(\langle \ZZ_{p},+\rangle)$ and the square of the lattice of all $(\ZZ_q,\ZZ_p)$-linearly closed clonoids.
\end{Thm}

\section{Preliminaries and notation}\label{Preliminaries3}

We use boldface letters for vectors, e. g., \index{$\mathbfsl{u}$}$\mathbfsl{u} = (u_1,\dots,u_n)$ for some $n \in \NN$. Moreover, we will use $\langle\mathbfsl{v}, \mathbfsl{u}\rangle$ for the scalar product of the vectors $\mathbfsl{v}$ and $\mathbfsl{u}$. Let $A$ be a set and let \index{$0_A \in A$}$0_A \in A$. We denote by \index{$\mathbf{0}_{A^n}$} \index{$\mathbf{0}_{n}$}$\mathbf{0}_{A^n}$ a constant $0_A$ vector of length $n$.

We denote by \index{$[n]$}$[n]$ the set $\{i \in \NN\mid 1 \leq i \leq n\}$ and by \index{$[n]_0$}$[n]_0$ the set $[n] \cup \{0\}$. Moreover we denote by \index{$\NN_0$}$\NN_0$ the set $\NN \cup \{0\}$. Let $\mathbfsl{x} \in \ZZ_p^n$ and let $\mathbfsl{a} \in [p-1]_0^n$. Then we denote by \index{$\mathbfsl{x}^{\mathbfsl{a}}$}$\mathbfsl{x}^{\mathbfsl{a}}$ the product $\prod_{i =1}^n(\mathbfsl{x})_i^{(\mathbfsl{a})_i}$.

From now one we will consider the group $\ZZ_p \times \ZZ_q$ instead of $\mathbb{Z}_{pq}$. We can see that the two groups are isomorphic and thus equivalent for our purpose. We see that the congruence lattice of $\ZZ_p \times \ZZ_q$ is the square lattice with the four congruences $\{0, 1, \pi_1 , \pi_2\}$. 

The reason why we consider $\ZZ_p \times \ZZ_q$ instead of $\mathbb{Z}_{pq}$ is that we want to distinguish the component $\ZZ_p$ from $\ZZ_q$ when we consider the domain or the codomain of a finitary function from $\ZZ_p \times \ZZ_q$ to itself.
Moreover, we consider $\ZZ_p^n \times \ZZ_q^n$ instead of $(\ZZ_p \times \ZZ_q)^n$ as domain of the functions we want to study. Let $f:\ZZ_p^n \times \ZZ_q^n \rightarrow \ZZ_p \times \ZZ_q$. We denote by $f_p:\ZZ_p^n \times \ZZ_q^n \rightarrow \ZZ_p$ the function $\pi_p^{\ZZ_p \times \ZZ_q} \circ f$ and by $f_q:\ZZ_p^n \times \ZZ_q^n \rightarrow \ZZ_q$  the function $\pi_q^{\ZZ_p \times \ZZ_q} \circ f$, where $\circ$ is the symbol for the composition of functions and $\pi_i^{\ZZ_p \times \ZZ_q}:\ZZ_p \times \ZZ_p \rightarrow \ZZ_{i}$ is the projection over $\ZZ_i$, with $i \in \{p,q\}$. Given a vector of variables $\mathbfsl{x}= (x_1,\dots,x_n)$ and $\mathbfsl{m} \in \ZZ_0^n$ we write \index{$\mathbfsl{x}^{\mathbfsl{m}}$}$\mathbfsl{x}^{\mathbfsl{m}}$ for $\prod_{i=1}^nx_i^{m_i}$.

Let $S$ be a set of finitary functions from a set $A$ to itself . We denote by \index{$\Clg(S)$}$\Clg(S)$ the \emph{clone generated by} $S$ on $A$. Let $p$ and $q$ be powers of distinct primes. We write \index{$\Cid(F)$}$\Cid(F)$ for the $(\FF_p,\FF_q)$-linearly closed clonoid generated by a set of functions $F$, as defined in \cite{Fio.CSOF}.

\section{Facts about clones}

Let $n \in \NN$. We denote by  $\mathbf{Clo}^{\mathcal{L}}(\langle\ZZ_{n},+\rangle)$ the lattice of all clones containing $\Clo(\langle \ZZ_{n},+\rangle)$.

In \cite{AM.PCOG} we can find a complete description of the polynomial clones (clones containing all constants) which contain $\Clo(\langle\ZZ_{pq},+\rangle)$. In figure $1$ we can see the lattice of all $17$ distinct polynomial clones containing $\Clo(\langle\ZZ_{pq},+\rangle)$, realized by P. Mayr.

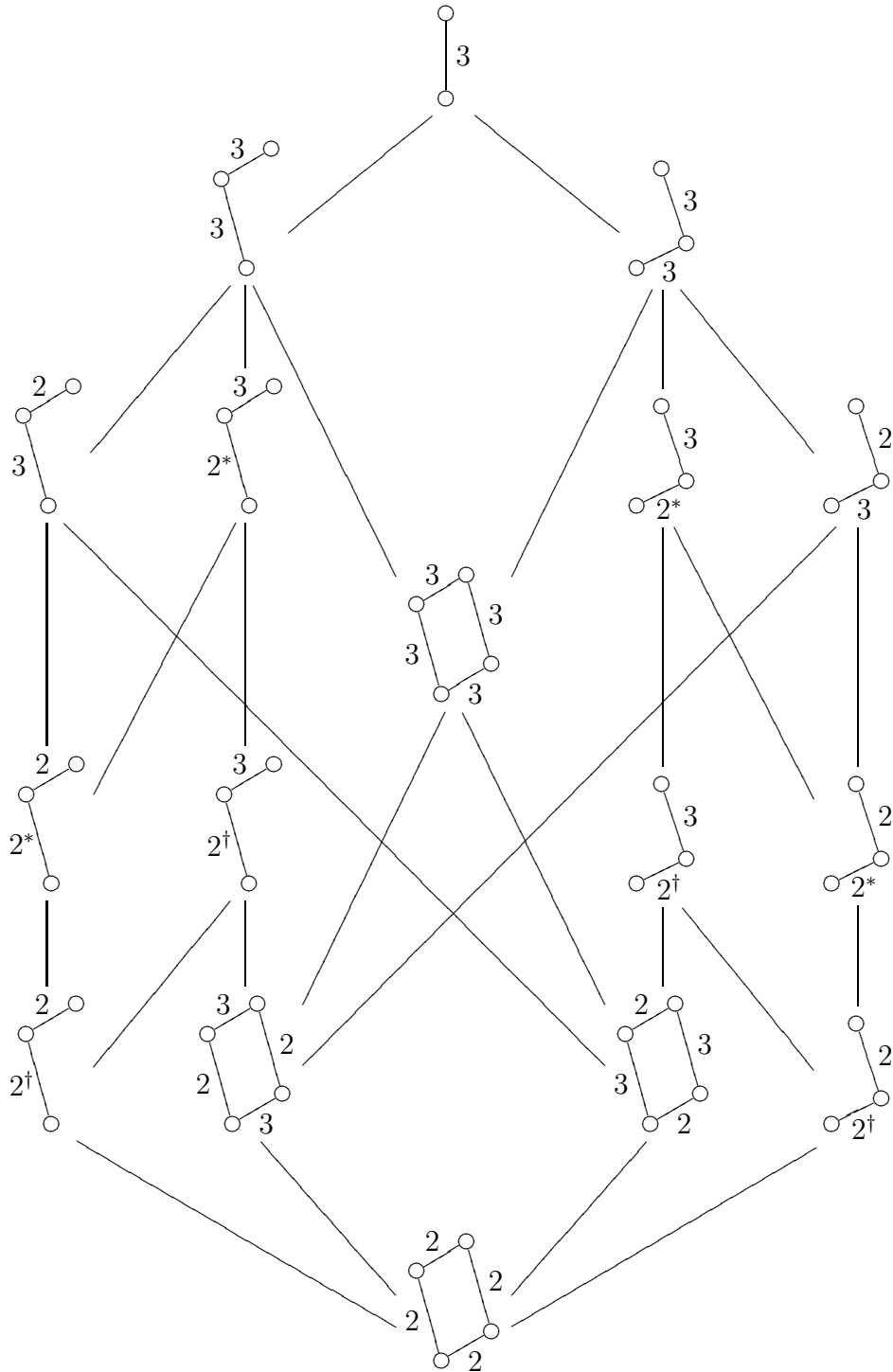
\begin{figure} \label{fig:clones}
	\centerline{ 
		\xymatrix{
			& & \latts\edge{dl}\edge{dr} & & \\
			& \lattl{3}{3}\edge{ddl}\edge{dd}\edge{dddr} & & \lattr{3}{3}\edge{dddl}\edge{dd}\edge{ddr} & \\
			& & & & \\
			\lattl{3}{2}\edge{dd}\edge{ddddrrr} & \lattl{2^*}{3} \edge{ddl}\edge{dd} & & \lattr{3}{2^*}\edge{dd}\edge{ddr} & \lattr{2}{3}\edge{dd}\edge{ddddlll}  \\ 
			& & \lattf{3}{3}\edge{dddl}\edge{dddr} & & \\
			\lattl{2^*}{2}\edge{dd} & \lattl{2^\dagger}{3}\edge{ddl}\edge{dd} & & \lattr{3}{2^\dagger}\edge{dd}\edge{ddr} & \lattr{2}{2^*}\edge{dd} \\ 
			& & & & \\
			\lattl{2^\dagger}{2}\edge{ddrr} & \lattf{2}{3}\edge{ddr} & & \lattf{3}{2}\edge{ddl} & \lattr{2}{2^\dagger}\edge{ddll} \\ 
			& & & & \\
			& & \lattf{2}{2} & & \\
	}}
	
	\caption{Polynomial clones containing $\Clo(\langle\mathbb{Z}_{pq}, +\rangle)$: Each clone $C$ is represented by its labelled congruence lattice. Simple factors are labelled 2 if they are abelian and 3 otherwise. A minimal factor which is labelled $2^\dagger$ is central; if it is labelled 2*, it is not central. (Picture and text realized by P. Mayr in \cite{AM.PCOG})}.
\end{figure}

Our goal is to describe also those clones that do not necessary contain constants. Indeed, \cite[Theorem $1.1$]{AM.PCOG} holds only for polynomial clones, and does not hold for clones that do not contain all the constants.

Let us now show some basic facts about finitary functions from $\ZZ_{pq}$ to $\ZZ_{pq}$ starting from the definition of affine map. 

\begin{Def}
	\label{DefAffin}	
	Let $\mathbf{G}$ be a group. Then an \emph{affine map} of $\mathbf{G}$ is an $n$-ary function $f$ which satisfies:
	
	\begin{center}
		
		for all $\mathbfsl{x}, \mathbfsl{y}, \mathbfsl{z} \in G^n$: \ $f(\mathbfsl{x} - \mathbfsl{y} + \mathbfsl{z}) = f(\mathbfsl{x}) - f(\mathbfsl{y}) + f(\mathbfsl{z})$.
		
	\end{center}
	
\end{Def}

\begin{Def}
	\label{DefAffin1nd2com}
	Let $n \in \NN$ and let $f:\ZZ_{p}^n \times \ZZ_{q}^n \rightarrow \ZZ_{p} \times \ZZ_{q}$ be a function. Then $f$ is an \emph{affine map in the first (second) component} if satisfies respectively:
	
	\begin{center}
		
		for all $\mathbfsl{x}_1, \mathbfsl{x}_2, \mathbfsl{x}_3 \in \ZZ_p^n$ and $\mathbfsl{y} \in \ZZ_q$: \\ $f(\mathbfsl{x}_1 - \mathbfsl{x}_2 + \mathbfsl{x}_3, \mathbfsl{y}) = f(\mathbfsl{x}_1, \mathbfsl{y}) - f(\mathbfsl{x}_2, \mathbfsl{y}) + f(\mathbfsl{x}_3, \mathbfsl{y})$
		
	\end{center}	
	and
	
	\begin{center}
		
		for all $\mathbfsl{x} \in \ZZ_q$ and $\mathbfsl{y}_1, \mathbfsl{y}_2, \mathbfsl{y}_3 \in \ZZ_p^n$: \\ $f(\mathbfsl{x}, \mathbfsl{y}_1 - \mathbfsl{y}_2 + \mathbfsl{y}_3) = f(\mathbfsl{x}, \mathbfsl{y}_1) - f(\mathbfsl{x}, \mathbfsl{y}_2) + f(\mathbfsl{x}, \mathbfsl{y}_3)$
		
	\end{center}
	for the second component.
	
\end{Def}

\begin{Rem}\label{RemPolCom}
	\theoremstyle{definition}
	It is a well-known fact that every finite field is polynomially complete. Thus for all $f:\FF_p^n \rightarrow \FF_p$, there exists a sequence $\{a_{\mathbfsl{m}}\}_{\mathbfsl{m} \in [p-1]_0^n} \subseteq \FF_p^n$ such that for all $\mathbfsl{x} \in \FF_p^n$, $f$ satisfies:
	\begin{equation}
	f(\mathbfsl{x}) = \sum_{\mathbfsl{m} \in [p-1]_0^n}a_{\mathbfsl{m}}\mathbfsl{x}^{\mathbfsl{m}}.
	\end{equation}
\end{Rem}

In the following remark we see how the affine maps from $\ZZ_p^n$ to $\ZZ_p$ are characterized.

\begin{Rem}\label{RemAffFun}
	\theoremstyle{definition}
	Let $f: \ZZ_p^n \rightarrow \ZZ_p$ be an $n$-ary function. Then $f$ is affine if and only if there exist $\mathbfsl{b} \in \ZZ_p^n$ and $c \in \ZZ_p$ such that for all $\mathbfsl{x} \in \ZZ_p^n$
	\begin{equation}
	f(\mathbfsl{x}) = \langle\mathbfsl{b}, \mathbfsl{x} \rangle + c.
	\end{equation}
	This holds because $f$ is affine if and only if for all $\mathbfsl{x}, \mathbfsl{y} \in \ZZ_p^n$, $f(\mathbfsl{x} + \mathbfsl{y}) = f(\mathbfsl{x}) + f(\mathbfsl{y}) + f(\mathbf{0}_n)$ if and only if $f = f' + c$, where $f'$ is a homomorphism from $\ZZ_p^n$ to $\ZZ_p$ and $c$ is a constant. Clearly, $f'$ is a homomorphism from $\ZZ_p^n$ to $\ZZ_p$ if and only if there exists $\mathbfsl{b} \in \ZZ_p^n$ such that for all $\mathbfsl{x} \in \ZZ_p^n$
	\begin{equation}
	f'(\mathbfsl{x}) = \langle\mathbfsl{b}, \mathbfsl{x} \rangle,
	\end{equation}
	since every homomorphism from $\ZZ_p$ to $\ZZ_p$ is a linear mapping. 
	
	Let $p$ and $q$ be distinct prime numbers. Let $f: \ZZ_p^n \rightarrow \ZZ_q$ be an $n$-ary function. With the same argument we can see that $f$ is affine if and only if $f$ is constant, since the only homomorphism from $\ZZ_p$ to $\ZZ_q$ is constant.
\end{Rem}

Next we introduce the relation $\rho(\alpha,\beta,\delta,m)$ following \cite[Definition $2.2$]{AM.PCOG}. 
\begin{Def}
	\label{DefRho}
	Let $\mathbf{A}$ be an algebra, $m : A^3 \rightarrow A$, and $\alpha, \beta, \gamma \in \Con(\mathbf{A})$. Then we define the $4$-ary relation $\rho(\alpha, \beta, \gamma, m)$ by:
	\begin{center}
		$\rho(\alpha, \beta, \gamma, m) := \{(a, b, c, d) \in A^4 \mid a\ \alpha\ b, b\ \beta\ c, m(a, b, c)\ \gamma\ d\}$.
	\end{center}
\end{Def}

 A Mal'cev polynomial of $A$ is a Mal'cev operation that lies in $\Pol_3(A)$.  Let $\A$ be an algebra  with a Mal'cev term $m$ and let $f \in \Clo(\A)$. In \cite[Proposition $2.3$]{AM.PCOG} and in many other sources we can find that for all $\alpha, \beta, \gamma \in \Con(\mathbf{A})$ we have that $[\alpha, \beta] \leq \gamma$ if and only if $\alpha$ centralizes $\beta$ modulo $\gamma$.

Let $f:A^n \rightarrow A$ be an $n$-ary function. We say that $f$ \emph{preserves} $R \in A^m$ if $(f((\mathbfsl{a}_1)_1,\dots,(\mathbfsl{a}_n)_1),\dots,f((\mathbfsl{a}_1)_m,\dots,(\mathbfsl{a}_n)_m)) \in R$ for all $\mathbfsl{a}_1,\dots,\mathbfsl{a}_n \in R$. Now we present a result in \cite[Lemma $2.4$]{AM.PCOG}, that allow us to connect the concepts of the relation $\rho$ and the centralizer relation.

\begin{Lem}
	\label{PropRho}
	Let $\mathbf{A}$ be an algebra in a congruence permutable variety, let $m$ be a Mal’cev polynomial on $\mathbf{A}$, and $\alpha, \beta, \gamma \in \Con(\mathbf{A})$. Then the following are equivalent:
	\begin{enumerate}
		\item[(1)] every $f \in \Pol(\mathbf{A})$ preserves $\rho(\alpha, \beta, \gamma, m)$.
		\item[(2)] $\alpha$ centralizes $\beta$ modulo $\gamma$.
	\end{enumerate}
	
\end{Lem}

Let $\mathbf{A}$ be an algebra in a congruence permutable variety, let $m$ be a Mal’cev polynomial on $\mathbf{A}$, and let $\alpha, \beta, \gamma \in \Con(\mathbf{A})$. Let $f:A^n \rightarrow A$ be an $n$-ary function. We say that $f$ \emph{preserves} $[\alpha,\beta] \leq \gamma$ if $f$ preserves $\rho(\alpha, \beta, \gamma, m)$.

From now on we will fix the algebra, $\ZZ_{pq}$, and the Mal'cev term $x-y+z$. Hence, for the sake of simplicity, we will consider $\rho$ as a ternary relation fixing the Mal'cev term of $\ZZ_{pq}$ as $x-y+z$.

With the following Lemma we show some properties of functions on $\ZZ_{pq}$ that preserve certain commutator relations.  

\begin{Lem}
	\label{Lem1-3}
	
	Let $p$ and $q$ be different prime numbers. Then for all $n \in \NN$ and $f: \ZZ_p^n \times \ZZ_q^n \rightarrow \ZZ_p \times \ZZ_q$ the following hold:
	
	\begin{enumerate}
		
		\item [(1)] $f$ preserves $\pi_1$  if and only if  there exist $f_p: \ZZ^n_p \rightarrow \ZZ_p$ and $f_q: \ZZ^n_p \times \ZZ_q^n \rightarrow \ZZ_q$ such that $f$ satisfies $f(\mathbfsl{x},\mathbfsl{y}) = (f_p(\mathbfsl{x}),f_q(\mathbfsl{x},\mathbfsl{y}))$ for all $(\mathbfsl{x},\mathbfsl{y}) \in \ZZ_p^n \times \ZZ_q^n$;
		
		\item [(2)] $f$ preserves $\pi_2$  if and only if there exist $f_p: \ZZ^n_p\times \ZZ_q^n \rightarrow \ZZ_p$ and  $f_q: \ZZ_q^n \rightarrow \ZZ_q$ such that $f$ satisfies $f(\mathbfsl{x},\mathbfsl{y}) = (f_p(\mathbfsl{x},\mathbfsl{y}),f_q(\mathbfsl{y}))$ for all $(\mathbfsl{x},\mathbfsl{y}) \in \ZZ_p^n \times \ZZ_q^n$;
		
		\item [(3)] suppose that $f$ preserves $\pi_1$. Then $f$ is affine in the second component if and only if $f$ preserves $[\pi_1,\pi_1] = 0$;
		
		\item [(4)]suppose that $f$ preserves $\pi_2$. Then $f$ is affine in the first component if and only if $f$ preserves $[\pi_2,\pi_2] = 0$; 
		
		\item [(5)] suppose that $f$ preserves $\pi_2$. Then $f_q|\{\mathbf{0}_n\} \times \ZZ_q^n$ is an affine function from $\ZZ_q^n$ to $\ZZ_q$ if and only if $f$ preserves $[1,1] \leq \pi_2$;
		
		\item [(6)] suppose that $f$ preserves $\pi_1$. Then $f_p |\ZZ_p^n \times \{\mathbf{0}_n\}$ is an affine function from $\ZZ_p^n$ to $\ZZ_p$ if and only if $f$ preserves $[1,1] \leq \pi_1$.
		
	\end{enumerate}
	
\end{Lem}

\begin{proof}
	
	Let us prove $(1)$. Let $n \in \NN$ and let $f: \ZZ_p^n \times \ZZ_q^n \rightarrow \ZZ_p \times \ZZ_q$. Then $f$ preserves $\pi_1$ if and only if for all $((\mathbfsl{x}_1,\mathbfsl{y}_1), (\mathbfsl{x}_2, \mathbfsl{y}_2)) \in (\ZZ_p \times \ZZ_q)^2$ whenever $\mathbfsl{x}_1 = \mathbfsl{x}_2$ then $\pi_p^{\ZZ_{p}\times \ZZ_q} \circ f(\mathbfsl{x}_1, \mathbfsl{y}_1) = \pi_p^{\ZZ_{p}\times \ZZ_q}\circ f(\mathbfsl{x}_2, \mathbfsl{y}_2)$, where $\pi_p^{\ZZ_{p}\times \ZZ_q} : \ZZ_p \times \ZZ_q \rightarrow \ZZ_p$ is the projection over $\ZZ_p$. This holds if and only if $f_p = \pi_p^{\ZZ_{p}\times \ZZ_q}\circ f$ depends on only the variables from $\ZZ_p$. Thus $(1)$ holds. The proof of item $(2)$ is symmetric to the one of item $(1)$.
	
	Next we prove $(3)$. Let $n \in \NN$ and let $f: \ZZ_p^n \times \ZZ_q^n \rightarrow \ZZ_p \times \ZZ_q$. By definition $f$ preserves $[\pi_1, \pi_1] = 0$ if and only if $f$ preserves $\rho(\pi_1, \pi_1, 0)$. Thus, by Definition \ref{DefRho}, $f$ preserves $\rho(\pi_1, \pi_1, 0)$ if and only if for all $\mathbfsl{x}_1, \mathbfsl{x}_2, \mathbfsl{x}_3, \mathbfsl{x}_4 \in \ZZ_p^n$ and $\mathbfsl{y}_1, \mathbfsl{y}_2, \mathbfsl{y}_3, \mathbfsl{y}_4 \in \ZZ_q^n$, $\mathbfsl{x}_1 = \mathbfsl{x}_2 = \mathbfsl{x}_3 = \mathbfsl{x}_4$ and $\mathbfsl{y}_1 - \mathbfsl{y}_2 + \mathbfsl{y}_3 = \mathbfsl{y}_4$ then:

	\begin{equation} \label{eq:4}
	f(\mathbfsl{x}_1, \mathbfsl{y}_1 - \mathbfsl{y}_2 + \mathbfsl{y}_3) = f(\mathbfsl{x}_1, \mathbfsl{y}_1) - f(\mathbfsl{x}_1, \mathbfsl{y}_2) + f(\mathbfsl{x}_1, \mathbfsl{y}_3).
	\end{equation}	
	Clearly, (\ref{eq:4}) holds if and only if $f$ is affine in the second component (Definition \ref{DefAffin1nd2com}). The proof of item $(4)$ is symmetric to the one of item $(3)$.
	
	Let us prove $(5)$. Let $n \in \NN$ and let $f: \ZZ_p^n \times \ZZ_q^n \rightarrow \ZZ_p \times \ZZ_q$. By definition $f$ preserves $[1, 1] \leq \pi_2$ if and only if $f$ preserves $\rho(1,1,\pi_2)$. Thus, by Definition \ref{DefRho}, $f$ preserves $\rho(1, 1, \pi_2)$ if and only if for all  $\mathbfsl{x}_1, \mathbfsl{x}_2, \mathbfsl{x}_3, \mathbfsl{x}_4 \in \ZZ_p^n$ and $\mathbfsl{y}_1, \mathbfsl{y}_2, \mathbfsl{y}_3, \mathbfsl{y}_4 \in \ZZ_q^n$ if $\mathbfsl{y}_1 - \mathbfsl{y}_2 + \mathbfsl{y}_3 = \mathbfsl{y}_4$ then 
	
	\begin{equation} \label{eq:5}
	f(\mathbfsl{x}_4, \mathbfsl{y}_1 - \mathbfsl{y}_2 + \mathbfsl{y}_3)\ \pi_2\ f(\mathbfsl{x}_1, \mathbfsl{y}_1) - f(\mathbfsl{x}_2, \mathbfsl{y}_2) + f(\mathbfsl{x}_3, \mathbfsl{y}_3).
	\end{equation}	
	By hypothesis $f$ preserves $\pi_2$ and thus, by item $(2)$, (\ref{eq:5}) holds if and only if $f_q = \pi_q^{\ZZ_p\times \ZZ_q} \circ f$ satisfies:	
	\begin{equation} \label{eq:6}
	f_q(\mathbfsl{y}_1 - \mathbfsl{y}_2 + \mathbfsl{y}_3) = f_q(\mathbfsl{y}_1) - f_q(\mathbfsl{y}_2) + f_q(\mathbfsl{y}_3),
	\end{equation}	
	for all $\mathbfsl{y}_1, \mathbfsl{y}_2, \mathbfsl{y}_3 \in \ZZ_q^n$. Hence this holds if and only if $f_q|\{\mathbf{0}_n\} \times \ZZ_q^n$ is an affine function. The proof of item $(6)$ is symmetric to the one of item $(5)$.

\end{proof}

In the case $p_1,\dots,p_m$ are distinct prime numbers we can see that we can split a function $f: \prod_{i=1}^m\ZZ_{p_i}^n \rightarrow \prod_{i=1}^m\ZZ_{p_i}$ in $f = \sum_{i=1}^m f_i$, where $f_i = \prod_{j\in [m]\backslash \{i\}}p_j^{p_i-1}f$. This implies, for example, that we can prove the following remark.

\begin{Rem}\label{RemLinComb}
	
	Let $p_1 \cdots p_m =s$ be a product of distinct prime numbers and let $C$ be a clone containing $\Clo(\langle \ZZ_{s}, +\rangle)$. Then for all $k \in \NN$ and $(\mathbfsl{a}_1,\dots,\mathbfsl{a}_m) \in \prod_{i=1}\ZZ_{p_i}^k$, $h_{(\mathbfsl{a}_1,\dots,\mathbfsl{a}_m)}:  \prod_{i=1}^m\ZZ_{p_i}^k \rightarrow  \prod_{i=1}^m\ZZ_{p_i}$ defined by $h_{(\mathbfsl{a}_1,\dots,\mathbfsl{a}_m)}: (\mathbfsl{x}_1,\dots,\mathbfsl{x}_m) \mapsto (\langle\mathbfsl{a}_1,\mathbfsl{x}_1\rangle,\dots, \langle$ $\mathbfsl{a}_m,\mathbfsl{x}_m\rangle)$ is in $C$.
	
\end{Rem}

\begin{proof}
	Let $p_1 \cdots p_m =s$ be a product of distinct prime numbers and let $C$ be a clone containing $\Clo(\langle \ZZ_{s}, +\rangle)$. Let $h_{p_i}: \prod_{i=1}\ZZ_{p_i} \rightarrow \prod_{i=1}\ZZ_{p_i}$ be such that $h_{p_i}:(x_1,\dots,x_m) \mapsto (0_{\ZZ_{p_1}}, \dots, \dots,0_{\ZZ_{p_{i-1}}},x_i ,0_{\ZZ_{p_{i+1}}},\dots,0_{\ZZ_{p_{m}}})$. Then $h = (\prod_{j\in [m]\backslash \{i\}}p_j)^{p_i-1}\pi_1^1$, where $\pi_1^1$ is the unary projection of $\Clo(\langle \ZZ_{pq}, +\rangle)$. Thus $\pi_1^1$ and the sum generate the function $h_{p_i}$ for all $i \in [m]$. Clearly, with $\{h_{p_i}\}_{i \in [m]}$, the sum, and the projections we can generate every other linear combination of the variables in the $m$ components.
\end{proof}

Let $A$ be a set and let $\FF_p$ be a field of order $p$. With the following lemma we show that every function from $\FF_p^n \times A^s$ to $\FF_p$ can be seen as the induced function of a polynomial of $\mathbf{R}[X]$, where $\mathbf{R} = \FF_p^{A^s}$. This easy fact will be often used later.

\begin{Lem}\label{Lem2Genexprofa}
	Let $A$ be a set and let $\FF_p$ be a field of order $p$. Then for every function $f$ from $\FF_p^n \times A^s$ to $\FF_p$ there exists a sequence of functions $\{f_{\mathbfsl{m}}\}_{\mathbfsl{m} \in [p-1]_0^n}$ from $A^s$ to $\FF_p$ such that $f$ satisfies for all $\mathbfsl{x} \in \FF_p^n$, $\mathbfsl{y} \in A^s$:
	
	\begin{equation}
	\label{GeneralExproffa}
	f(\mathbfsl{x}, \mathbfsl{y}) = \sum_{\mathbfsl{m} \in [p-1]_0^n} f_{\mathbfsl{m}}(\mathbfsl{y})\mathbfsl{x}^{\mathbfsl{m}}.
	\end{equation}

\end{Lem}

\begin{proof}
	Let $f$ be a function from $\FF_p^n \times A^s$ to $\FF_p$ and let us fix the variables over $A^s$ as a vector $\mathbfsl{a} \in A^s$ and let $f_{\mathbfsl{a}}: \FF_p^n \rightarrow \FF_p$ be defined by $f_{\mathbfsl{a}}: \mathbfsl{x} \mapsto f(\mathbfsl{x}, \mathbfsl{a})$. By Remark \ref{RemPolCom}, we have that for every $\mathbfsl{a} \in A^s$ there exists $\{c_{(\mathbfsl{m},\mathbfsl{a})}\}_{\mathbfsl{m} \in [p-1]_0^n}$ such that for all $\mathbfsl{x} \in \FF_p^n$ $f$ satisfies:
	
	\begin{equation}
	f(\mathbfsl{x}, \mathbfsl{a}) = f_{\mathbfsl{a}}(\mathbfsl{x}) = \sum_{\mathbfsl{m} \in [p-1]_0^n}c_{(\mathbfsl{m},\mathbfsl{a})}\mathbfsl{x}^{\mathbfsl{m}}.
	\end{equation}
	Thus we can see that for all $\mathbfsl{x} \in \ZZ_p^n$, $\mathbfsl{y} \in \ZZ_p^n$: 
	
	\begin{equation}
	f(\mathbfsl{x}, \mathbfsl{y}) = \sum_{\mathbfsl{a} \in A^s} \chi_{\mathbfsl{a}}(\mathbfsl{y}) \sum_{\mathbfsl{m} \in [p-1]_0^n}c_{(\mathbfsl{m},\mathbfsl{a})}\mathbfsl{x}^{\mathbfsl{m}},
	\end{equation}
	where $\chi_{\mathbfsl{a}}(\mathbfsl{a}) = 1$ and $\chi_{\mathbfsl{a}}(\mathbfsl{y}) = 0$, for all $\mathbfsl{y} \in A^s\backslash\{\mathbfsl{a}\}$. Thus for all $\mathbfsl{m} \in [p-1]_0^n$ we define $f_{\mathbfsl{m}}: A^s \rightarrow \FF_p$ by:

	\begin{equation}
	f_{\mathbfsl{m}} := \sum_{\mathbfsl{a} \in A^s} c_{(\mathbfsl{m}, \mathbfsl{a})}\chi_{\mathbfsl{a}}.
	\end{equation}
	Thus the claim holds.

\end{proof}

The previous lemma in our setting implies the following.

\begin{Lem}\label{Lem2Genexprof}
	Let $p$ and $q$ be prime numbers. Then for every function $f$ from $\ZZ_p^n \times \ZZ_q^n$ to $\ZZ_p \times \ZZ_q$ there exist two sequences of functions $\{f_{\mathbfsl{m}}\}_{\mathbfsl{m} \in [p-1]_0^n}$ from $\ZZ_q^n$ to $\ZZ_p$  and $\{f_{\mathbfsl{h}}\}_{\mathbfsl{h} \in [q-1]_0^n}$ from $\ZZ_p^n$ to $\ZZ_q$ such that $f$ satisfies for all $\mathbfsl{x} \in \ZZ_p^n$, $\mathbfsl{y} \in \ZZ_q^n$:
	
	\begin{equation}
	\label{GeneralExproff}
	f(\mathbfsl{x}, \mathbfsl{y}) = (\sum_{\mathbfsl{m} \in [p-1]_0^n} f_{\mathbfsl{m}}(\mathbfsl{y})\mathbfsl{x}^{\mathbfsl{m}}, \sum_{\mathbfsl{h} \in [q-1]_0^n} f_{\mathbfsl{h}}(\mathbfsl{x})\mathbfsl{y}^{\mathbfsl{h}}).
	\end{equation}

\end{Lem}

\section{Embedding of the Clonoids}
\label{EmbClo}

The aim of this section is to prove that there exists an embedding of the lattice of all $(\mathbb{Z}_p,\ZZ_q)$-linearly closed clonoids in the lattice of all clones containing $\Clo(\langle \ZZ_{pq},+\rangle)$, when $p$ and $q$ are distinct prime numbers. 

Let $e: \bigcup_{n\in \NN}\ZZ_p^{\ZZ_q^n} \rightarrow \bigcup_{n\in \NN}(\ZZ_p \times \ZZ_q)^{\ZZ_p^n \times \ZZ_q^n}$ be such that for all $n \in \NN$ and for all $f \in \ZZ_p^{\ZZ_q^n}$, $e(f) := g$, where $g: \ZZ_p^n \times \ZZ_q^n \rightarrow \ZZ_p \times \ZZ_q$ satisfies for all $(\mathbfsl{x},\mathbfsl{y}) \in \ZZ_p^n \times \ZZ_q^n$, $g(\mathbfsl{x},\mathbfsl{y}) = (f(\mathbfsl{y}),0)$.

Furthermore, we define $\gamma$ from the lattice of all $(\mathbb{Z}_p,\ZZ_q)$-linearly closed clonoids to the lattice of all clones containing $\Clo(\langle \ZZ_{pq},+\rangle)$ such that for all $C \in \mathcal{L}(\ZZ_p,\ZZ_q)$:

\begin{equation} \label{equ:1}
\gamma(C) := \{f \mid f: \ZZ_p^n \times \ZZ_q^n \rightarrow \ZZ_p \times \ZZ_q, \exists (g \in C, \mathbfsl{a} \in \ZZ_p^n, \mathbfsl{b} \in \ZZ_q^n): f = e(g)+ h_{(\mathbfsl{a},\mathbfsl{b})}\},
\end{equation}
where $h_{(\mathbfsl{a},\mathbfsl{b})}:\ZZ_p^n \times \ZZ_q^n \rightarrow \ZZ_p \times \ZZ_q$ is defined in Remark \ref{RemLinComb}.

In order to prove Theorem \ref{ThEmbClonoids} we first show a lemma.

\begin{Lem}
	\label{Lemclogen-2}
	Let $p$ and $q$ be powers of distinct primes. Let $X \subseteq \bigcup_{n \in \NN} \mathbb{F}_p^{\mathbb{F}_q^n}$. Then $\Cid(X) = \bigcup_{n \in \NN} X_n$ where:
	\begin{flushleft}
		$X_0 = X$
		\\$X_{n+1} = \{af + bg \mid a,b \in \mathbb{F}_p, f,g \in X_n^{[r]}, r \in \NN\} \cup \{g: (x_1,\dots,x_l) \mapsto f(A\cdot (x_1,\dots,x_l)^t) \mid f \in X_n^{[k]}, A \in \mathbb{F}^{k \times l}_{q}\}$.
	\end{flushleft}
\end{Lem}

\begin{proof}
	Let $X$ be as in the hypothesis. Then we show that $\Cid(X) = \bigcup_{n \in \NN} X_n$. The $\supseteq$ inclusion is obvious by Definition \ref{DefClo}. For the other inclusion we prove that $S = \bigcup_{n \in \NN} X_n$ is an $(\mathbb{F}_p,\mathbb{F}_q)$-linearly closed clonoid. 
	
	Let $f, g \in S^{[r]}$. Then there exist $n_1, n_2 \in \NN$ such that $f \in X_{n_1}$ and $g \in X_{n_2}$. Let $n_3 = max(\{n_1,n_2\})$ then $f, g \in X_{n_3}$ and $af+bg \in X_{n_3+1}$ for all $a,b \in \FF_p$.
	Furthermore, let $h \in S^{[k]}$. Then there exists $n' \in \NN$ such that $h \in X_{n'}$. Let  $A \in \mathbb{F}_q^{k \times r}$. By definition of $X_{n'+1}$, we have that the function $g: \mathbfsl{x} \mapsto h(A\cdot \mathbfsl{x})$, for all $\mathbfsl{x}\in \FF^r_{q}$,  is in $X_{n'+1}$, and hence $S$ is an $(\mathbb{F}_p,\mathbb{F}_q)$-linearly closed clonoid. 
	Furthermore $S \supseteq X$ and thus $S =\Cid(X)$, which is the smallest $(\mathbb{F}_p,\mathbb{F}_q)$-linearly closed clonoid containing $X$. Thus the claim holds.
	
\end{proof}

\begin{proof}[Proof of Theorem \ref{ThEmbClonoids}]

	Let $\gamma$ be the function defined in (\ref{equ:1}). First we show that $\gamma$ is well-defined. Let $C$ be a $(\ZZ_p,\ZZ_q)$-linearly closed clonoid. Then $\gamma(C)$ contains the projections and the binary sum of $\ZZ_p \times \ZZ_q$. Moreover, let $f,f_1,\dots,f_n \in \gamma(C)$ be respectively an $n$-ary and $n$ $m$-ary functions. Then there exist $g_f,g_1,\dots,g_n \in C$, $\mathbfsl{a}_f \in \ZZ_p^n$, $\mathbfsl{b}_f \in \ZZ_q^n$, $\mathbfsl{a}_1,\dots,\mathbfsl{a}_n \in \ZZ_p^m$, and $\mathbfsl{b}_1,\dots,\mathbfsl{b}_n \in \ZZ_q^m$ such that:
	
	\begin{center}
		$f(\mathbfsl{x},\mathbfsl{y}) = (\langle \mathbfsl{a}_f,\mathbfsl{x} \rangle + g_f(\mathbfsl{y}), \langle \mathbfsl{b}_f,\mathbfsl{y} \rangle )$,
	\end{center}
	for all $\mathbfsl{x} \in \ZZ_p^n$, $\mathbfsl{y} \in \ZZ_q^n$, and for all $1 \leq i \leq n$:
	
	\begin{center}
		$f_i(\mathbfsl{x},\mathbfsl{y}) = (\langle \mathbfsl{a}_i,\mathbfsl{x} \rangle + g_i(\mathbfsl{y}), \langle \mathbfsl{b}_i,\mathbfsl{y} \rangle )$,
	\end{center}
	for all $\mathbfsl{x} \in \ZZ_p^m$, $\mathbfsl{y} \in \ZZ_q^m$. Then $h = f \circ (f_1,\dots,f_n)$ can be written as:
	
	\begin{center}
		$h(\mathbfsl{x},\mathbfsl{y}) = (\langle \mathbfsl{a}_h,\mathbfsl{x} \rangle + g_h(\mathbfsl{y}), \langle \mathbfsl{b}_h,\mathbfsl{y} \rangle )$,
	\end{center}	
	where for all $j=1, \dots,m$, $(\mathbfsl{a}_h)_j = \sum_{i =1}^n(\mathbfsl{a}_f)_i(\mathbfsl{a}_i)_j$, $(\mathbfsl{b}_h)_j = \sum_{i =1}^n(\mathbfsl{b}_f)_i(\mathbfsl{b}_i)_j$, and $g_h: \ZZ_q^m \rightarrow \ZZ_p$ is defined by $g_h(\mathbfsl{y}) = \langle \mathbfsl{a}_f,\mathbf{d}(\mathbfsl{y}) \rangle + g_f(\langle \mathbfsl{b}_1, \mathbfsl{y}\rangle, \dots, \langle \mathbfsl{b}_n, \mathbfsl{y}\rangle)$ with $\mathbf{d}(\mathbfsl{y}) = (g_1(\mathbfsl{y}),\dots,g_n(\mathbfsl{y}))$,
	for all $\mathbfsl{y} \in \ZZ_q^m$. We can see from Definition \ref{DefClo} that $g_h \in C$. Thus $\gamma(C)$ is closed under composition and $\gamma$ is well-defined. 
	
	Next we prove that $\gamma$ is injective. Let $C$ and $D$ be two $(\ZZ_p,\ZZ_q)$-linearly closed clonoids such that $\gamma(C) = \gamma(D)$ and let $g \in C$ be an $n$-ary function. Then let $s: \ZZ_p^n \times \ZZ_q^n \rightarrow \ZZ_p \times \ZZ_q$ be such that $e(g) =s$. Then $s$ is in $\gamma(C) = \gamma(D)$. By definition of $\gamma$, this implies that $e(g) = e(g') + h_{(\mathbfsl{a},\mathbfsl{b})}$ for some $g' \in D$ and $(\mathbfsl{a},\mathbfsl{b}) \in \ZZ_p^n \times \ZZ_q^n$. The only possibility is that  $g = g' \in D$ and thus $C \subseteq D$. We can repeat this argument for the other inclusion and hence $\gamma$ is injective. Furthermore, we have that for all $C,D \in \mathcal{L}(\ZZ_p,\ZZ_q)$ clearly $\gamma(C \cap D) = \gamma(C) \cap \gamma(D)$. We can observe that $\gamma$ is monotone, thus $\gamma(C \vee D) \supseteq \gamma(C) \vee \gamma(D)$. For the other inclusion we prove by induction on $n$ that $\gamma(C) \vee \gamma(D) \supseteq e(X_n)$, where $C \vee D = \bigcup_{n \in \NN} X_n$ with:
	\begin{flushleft}
		$X_0 = C \cup D$
		\\$X_{n+1} = \{af + bg \mid a,b \in \mathbb{F}_p, f,g \in X_n^{[m]}, m \in \NN\} \cup \{g : (x_1,\dots,x_s) \mapsto f(A\cdot (x_1,\dots,x_s)^t) \mid f \in X_n^{[m]}, A \in \mathbb{Z}_q^{m \times s}, m,s \in \NN \}$.
	\end{flushleft}
	Base step $n = 0$: $e(C \cup D) = e(C) \cup e(D) \subseteq \gamma(C) \vee \gamma(D)$.
	
	Induction step $n >0$: suppose that the claim holds for $n-1$. Then let $g \in e(X_{n})$. Hence there exists $u \in X_n$ such that $e(u)=g$ and either $u$ is a linear combination of functions in $X_{n-1}$ or there exist $f \in  X_{n-1}^{[m]}$, $A \in \mathbb{Z}_q^{m \times s}$, and $m,s \in \NN$ such that $u : \mathbfsl{x} \mapsto f(A\cdot \mathbfsl{x}^t)$. In both cases we have $g \in \Clg(e(X_{n-1})) \subseteq \gamma(C) \vee \gamma(D)$ and this concludes the induction proof. By Lemma \ref{Lemclogen-2}, $\gamma(C) \vee \gamma(D) \supseteq e(C \vee D)$. 
	
	Let $f \in \gamma(C \vee D)^{[n]}$. Then there exist $g \in C \vee D$,  $\mathbfsl{a} \in \ZZ_p^n$, and $\mathbfsl{b} \in \ZZ_q^n$ such that $f = e(g) + h_{(\mathbfsl{a},\mathbfsl{b})}$. By Remark \ref{RemLinComb}, $h_{(\mathbfsl{a},\mathbfsl{b})} \in \gamma(C) \vee \gamma(D)$, thus $f \in \gamma(C) \vee \gamma(D)$ and $\gamma(C) \vee \gamma(D) = \gamma(C \vee D)$. Hence $\gamma$ is an embedding.
\end{proof}

\section{Independent algebras}
\label{SecIndipAlg}
In this section we characterize all clones containing $\Clo(\langle\ZZ_{pq},+\rangle)$ that preserve the whole congruence lattice of four elements. In this cases we see that we can actually split these clones on $\ZZ_{pq}$ in two clones, one containing $\Clo(\langle\ZZ_{p},+\rangle)$ and the other containing $\Clo(\langle\ZZ_{q},+\rangle)$. To this end, let us introduce the concept of independent algebras (see also \cite{AM.IOAW}).

\begin{Def}
	
	Two algebras $\A$ and $\B$ of the same variety $\vv{V}$ are \emph{independent} if there exists a binary term in the language of $\vv{V}$ such that $ \A \models t(x,y) \approx x$ and $ \B \models t(x,y) \approx y$.	
	
\end{Def}

Now we prove a theorem that characterizes the clones containing $\Clo(\langle \ZZ_{pq},+\rangle)$ whose algebra can be split in a direct product of two independent algebras. We denote by \index{$\mathbf{Clo}^{\diamond}$}$\mathbf{Clo}^{\diamond}(\langle\ZZ_{pq},+\rangle)$ the lattice of all clones containing $\Clo(\langle\ZZ_{pq},+\rangle)$ which preserve the congruences $\{0, \pi_1, \pi_2, 1\}$. We will present an independent proof of Theorem \ref{ThmembeddingClones} which also follows from \cite[Lemma $6.1$]{AM.IOAW}.

\begin{proof}[Proof of Theorem \ref{ThmembeddingClones}]
	
	Let $\rho_1: \mathbf{Clo}^{\diamond}(\langle\ZZ_{pq},+\rangle) \rightarrow  \mathbf{Clo}^{\mathcal{L}}(\langle\ZZ_{p},+\rangle)$ and $\rho_2:  \mathbf{Clo}^{\diamond}(\langle$ $\ZZ_{pq},+\rangle)$ $ \rightarrow  \mathbf{Clo}^{\mathcal{L}}(\langle\ZZ_{q},+\rangle)$ be defined by:
	
	\begin{equation}
	\begin{split}
	\rho_1(C) &:= \{f \mid \text{ there exist } n \in \NN, g \in C^{[n]}: f = \pi_p^{\ZZ_p \times \ZZ_q} \circ  g\mid{\ZZ^n_p \times \{\mathbf{0}_n \}}\}
	\\\rho_2(C) &:= \{f \mid \text{ there exist } n \in \NN, g \in C^{[n]}: f = \pi_q^{\ZZ_p \times \ZZ_q} \circ g\mid{\{\mathbf{0}_n\} \times \ZZ^n_q} \},
	\end{split}
	\end{equation}
	where $\pi_p^{\ZZ_p \times \ZZ_q},\pi_q^{\ZZ_p \times \ZZ_q}$ are respectively the projections over $\ZZ_p$ and over $\ZZ_q$. Let $\rho: \Clo(\langle\ZZ_{pq},+\rangle) \rightarrow \Clo(\langle\ZZ_{p},+\rangle) \times \Clo(\langle\ZZ_{q},+\rangle)$ be defined by $\rho: C \mapsto (\rho_1(C),\rho_2(C))$ for all $C \in \Clo(\langle\ZZ_{pq},+\rangle)$. First of all let us prove that $\rho$ is well-defined. To this end, let $C \in \Clo(\langle\ZZ_{pq},+\rangle)$.  By $(1)$ and $(2)$ of Lemma \ref{Lem1-3}, for all $n \in \NN$ and $f \in C^{[n]}$ there exist $f_p:\ZZ^n_p \rightarrow \ZZ_p$ and $f_q:\ZZ^n_q \rightarrow \ZZ_q$ such that for all $(\mathbfsl{x},\mathbfsl{y}) \in \ZZ_p^n \times \ZZ_q^n$, $f(\mathbfsl{x},\mathbfsl{y}) = (f_p(\mathbfsl{x}),f_q(\mathbfsl{y}))$. Furthermore, let $n \in  \NN$. By Remark \ref{RemLinComb}, we have that the function $h: \ZZ_p^n \times \ZZ_q^n \rightarrow \ZZ_p \times \ZZ_q$ defined by $h: (\mathbfsl{x},\mathbfsl{y}) \mapsto ((\mathbfsl{x})_i,(\mathbfsl{y})_i)$ is in $C$ and thus the projections $\pi_i^n \in \rho_1(C)$, for all $1 \leq i \leq n$. Let $f \in \rho_1(C)^{[n]}$ and $g_1,\dots,g_n \in \rho_1(C)^{[m]}$. Then there exist $f',g_1',\dots,g'_n \in C$ such that $f = \pi_p^{\ZZ_p \times \ZZ_q} \circ f'\mid{\ZZ^n_p \times \{\mathbf{0}_n \}}$ and for all $1 \leq i \leq n$, $g_i= \pi_p^{\ZZ_p \times \ZZ_q} \circ g_i'\mid{\ZZ^m_p \times \{\mathbf{0}_m \}}$. Thus we have that $f' \circ (g'_1,\dots,g'_n) \in C$. Hence $f \circ (g_1,\dots,g_n)= \pi_p^{\ZZ_p \times \ZZ_q} \circ  f'\circ (g'_1,\dots,g'_n)\mid\ZZ^m_p \times \{\mathbf{0}_m \} \in \rho_1(C)$. Furthermore, the binary sum of $\ZZ_p$ is in $C$, since, by Remark \ref{RemLinComb}, the  function $h:\ZZ_p^2 \times \ZZ_q^2 \rightarrow \ZZ_p \times \ZZ_q$ defined by $h:((x_1,y_1),(x_2,y_2)) \mapsto (x_1+x_2,0)$ is in $C$. Hence $\rho_1(C) \in \Clo(\langle\ZZ_{p},+\rangle)$ and $\rho_1$ is well-defined. Symmetrically, we can prove that $\rho_2$ is well-defined.
	
	We define the function $\psi: \Clo^{\mathcal{L}}(\langle\ZZ_{p},+\rangle) \times \Clo^{\mathcal{L}}(\langle\ZZ_{q},+\rangle) \rightarrow  \Clo^{\diamond}(\langle\ZZ_{pq},+\rangle)$ such that $\psi(C_1,C_2) := \{f \mid \text{ there exist }  n \in \NN, f_1 \in C_1^{[n]},  f_2 \in C_2^{[n]}: \text{ for all } (\mathbfsl{x}, \mathbfsl{y}) \in \ZZ_p^n \times \ZZ_q^n, f(\mathbfsl{x}, \mathbfsl{y}) = (f_1(\mathbfsl{x}),f_2(\mathbfsl{y}))\}$.
	
	Clearly, $\psi$ is well-defined. We prove that $\rho \circ \psi = \id_{\Clo^{\mathcal{L}}(\langle\ZZ_{p},+\rangle)\times \Clo^{\mathcal{L}}(\langle\ZZ_{q},+\rangle) }$ and $\psi \circ \rho = \id_{\Clo^{\diamond}(\langle\ZZ_{pq},+\rangle)}$. First we can see that both $\rho$ and $\psi$ are monotone.
	
	For the first inequality let $(C_1,C_2) \in \Clo^{\mathcal{L}}(\langle\ZZ_{p},+\rangle)\times \Clo^{\mathcal{L}}(\langle\ZZ_{q},+\rangle) $. For the monotonicity of both the functions we have that $\rho \circ \psi ((C_1,C_2) ) \supseteq (C_1,C_2)$. Let $(f_1,f_2) \in \rho \circ \psi ((C_1,C_2) )^{[n]}$. Then there exist $g_1, g_2 \in \psi ((C_1,C_2))$ such that $f_1 = \pi_p^{\ZZ_p \times \ZZ_q}\circ g_1\mid{\ZZ^n_p \times \{\mathbf{0}_n\} }  $ and $f_2 = \pi_q^{\ZZ_p \times \ZZ_q} \circ g_2\mid{\{\mathbf{0}_n\} \times \ZZ^n_q} $. Thus, by definition of $\psi$, $(f_1,f_2) \in (C_1,C_2)$ and the first inequality holds.
	
	For the second inequality let $C \in \Clo^{\diamond}(\langle\ZZ_{pq},+\rangle)$. For the monotonicity of both $\psi$ and $\rho$  we have that $\psi \circ \rho (C) \supseteq C$. Let $f \in \psi \circ \rho (C)$. Then there exists $(f_1, f_2) \in \rho(C)$ such that for all $(\mathbfsl{x}, \mathbfsl{y}) \in \ZZ_p^n \times \ZZ_q^n$, $f(\mathbfsl{x}, \mathbfsl{y}) = (f_1(\mathbfsl{x}),f_2(\mathbfsl{y}))$. Then there exist $g_1, g_2 \in C$ such that $f_1 = \pi_p^{\ZZ_p \times \ZZ_q} \circ  g_1\mid{\ZZ_p^n \times \{\mathbf{0}_n \}}$ and $f_2 = \pi_q^{\ZZ_p \times \ZZ_q} \circ g_2\mid{\{\mathbf{0}_n\} \times \ZZ^n_q}$.   Thus $f = q^{(p-1)}g_1 + p^{(q-1)}g_2\in C$ and the second inequality holds.

	Thus $\psi = \rho^{-1}$ and $\rho$ are monotone bijective functions with monotone inverse and hence $\rho$ is a lattice isomorphism.

\end{proof}

\begin{Cor}
	
	Let $p$ and $q$ be distinct prime numbers. Then there are $(n(q) + 3)(n(p) + 3)$ many clones containing $\Clo(\langle\ZZ_{pq},+\rangle)$ which preserve the congruences $\{0, \pi_1, \pi_2, 1\}$, where $n(k)$ is the number of divisors of $k-1$. 

\end{Cor}

\begin{proof}	
	From Theorem \ref{ThmembeddingClones} we obtain that the number of clones containing $\Clo(\langle\ZZ_{pq},$ $+\rangle)$ is the product of the number of clones containing $\Clo(\langle\ZZ_{p},+\rangle)$ and the number of clones containing $\Clo(\langle\ZZ_{q},+\rangle)$. These numbers are determined in \cite[Cororllary $1.3$]{Kre.CFSO}, for $p > 2$ and in \cite{Pos.TTVI} for $p = 2$. From these two results we have that for every prime number $p$ the clones containing $\Clo(\langle\ZZ_{p},+\rangle)$ are $n(p) + 3$, where $n(k)$ is the number of divisors of $k-1$. 

\end{proof}

This complete the characterization of all clones containing $\Clo(\langle\ZZ_{pq},+\rangle)$ which belong to products of independent algebras.

\section{A general bound}
\label{AGenBoun}

In the following section our goal is to determine a bound for the cardinality of the lattice of all clones containing $\Clo(\langle \ZZ_{pq},+\rangle)$ and to find a bound for the arity of  sets of generators of these clones. Theorem \ref{Cor3} gives a complete list of generators for a clone containing $\Clo(\langle\ZZ_{pq},+\rangle)$ that explains the connection between clonoids and clones in this case. The generators of Theorem \ref{Cor3} are substantially formed by a product of a unary function generating a $(\ZZ_p,\ZZ_q)$-linearly closed clonoid and a monomial generating a clone on $\ZZ_p$. This puts together the characterization in \cite{Kre.CFSO} and \cite[Theorems $1.2$, $1.3$]{Fio.CSOF}, which are the main ingredients of this section.

 We start showing some lemmata which we need to prove that clones containing $\Clo(\langle \ZZ_{pq},+\rangle)$ are strictly characterized by the $(\ZZ_p,\ZZ_q)$-linearly closed clonoids and the $(\ZZ_q,\ZZ_p)$-linearly closed clonoids.

In this paper we have to deal with polynomials whose coefficients are finitary functions from $\ZZ_q$ to $\ZZ_p$. From now on let $\mathbf{R}[X]$ be a polynomial ring where $\mathbf{R} = \ZZ_p^n$ for some $n \in \NN$. The next step will be to generalize some results in \cite{Kre.CFSO} about $p$-linearly closed clonoids to polynomials in \index{$\mathbf{R}[X]$}$\mathbf{R}[X]$. Let us start with the notation.

Following \cite{Kre.CFSO} we denote by \index{$\tD(h)$}$\tD(h)$ the \emph{total degree of a monomial} $h$, which is defined as the sum of the exponents. We denote by \index{$\Var(h)$}$\Var(h)$ \emph{the set of variables occurring in} $h$ and with $\Var(f)$ \emph{the set of variables occurring in monomials of} $f$ with non-zero coefficients. We denote by \index{$\MInd(f)$}$\MInd(f)$ the maximum of the indices of variables in $\Var(f)$. Let $f \in \mathbf{R}[X]$ be such that $\MInd(f) = n$ and let $\mathbfsl{x}$ be a vector composed by the variables in $(x_1,\dots,x_n)$. Then $f$ can be written as:

\begin{equation}
\label{expressiongen}
f =\sum_{\mathbfsl{m} \in \NN_0^n}r_{\mathbfsl{m}}\mathbfsl{x}^{\mathbfsl{m}},
\end{equation}
for some sequence $\{r_{\mathbfsl{m}}\}_{\mathbfsl{m} \in \NN_0^n}$ in $\mathbf{R}$ with only finitely many non-zero members and where $\mathbfsl{x}^{\mathbfsl{m}}=\prod_{i=1}^nx_i^{m_i}$

We denote the set of all \emph{degrees of} $f$ by \index{$\DEGS(f)$}$\DEGS(f) := \{\mathbfsl{m} \mid \mathbfsl{m} \in \NN_0^n, r_{\mathbfsl{m}} \not= 0\}$ with $\DEGS(0) := \emptyset$. We denote by \index{$\MON(f)$}$\MON(f)$ the set of all monomials in $f$. We denote by \index{$\TD(f)$}$\TD(f) := \{\langle \mathbfsl{m}, \mathbf{1}_n \rangle \mid \mathbfsl{m} \in \DEGS(f)\}$.

Next we introduce a notation for the composition of multivariate polynomials. Let $m,n,h \in \NN$, $g,f_1,\dots,f_h \in \mathbf{R}[X]$, and let $\mathbfsl{b} = (b_1, \dots, b_h ) \in \NN^h$ with $1 \leq b_1 < b_2 < \cdots < b_h \leq \MInd(f)$. Then we define \index{$\circ_{\mathbfsl{b}}$}$g \circ_{\mathbfsl{b}} (f_1 ,\dots, f_m)$ by:

\begin{equation}
g \circ_{\mathbfsl{b}} (f_1,\dots,f_h) := g(x_1,\dots, x_{b_1-1}, f_1, x_{b_1+1},\dots,x_{b_2-1}, f_2, x_{b_2+1},\dots ).
\end{equation}
Let $p$ be a prime and let $f \in \mathbf{R}[X]$, where $\mathbf{R} = \ZZ_p^n$. Since later we want to introduce the induced function of a polynomial, in order to have a unique polynomial for every induced function, we consider the ideal $I$ generated by the polynomials $x_i^p-x_i \in \mathbf{R}[X]$ in $\mathbf{R}[X]$, for every $x_i \in X$. By \cite[Chapter $15.3$]{Eis.CA} there is a unique reminder $rem(f)$ of $f$ with respect to $I$. This remainder has the property that the exponents of the variables are less or equal $p-1$. Following \cite[Section $2$]{Kre.CFSO}, we denote by $\mathbf{R}^*[X] $ the set of all polynomials in $\mathbf{R}[X]$ with this property. We can observe that these polynomials form a set of representatives of the set of all classes of the quotient $\mathbf{R}/I$.

Let $f \in \mathbf{R}^*[X]$ with $\MInd(f) = n$ and let $\mathbfsl{x}$ be a vector composed by the variables in $(x_1,\dots,x_n)$. Then $f$ can be written as:

\begin{equation}
\label{expressiongenmodp}
f =\sum_{\mathbfsl{m} \in [p-1]_0^n}r_{\mathbfsl{m}}\mathbfsl{x}^{\mathbfsl{m}},
\end{equation}
for some sequence $\{r_{\mathbfsl{m}}\}_{\mathbfsl{m} \in [p-1]_0^n}$ in $\mathbf{R}$. 

\begin{Def}
	\label{DefPolyClonoid}
	Let $p$ be a prime number and let $\mathbf{R} = \ZZ_p^n$. An $\mathbf{R}$-\emph{polynomial linearly closed clonoid} is a non-empty subset $C$ of $\mathbf{R}[X]$ with the following properties:
	
	\begin{enumerate}
		\item[(1)] for all $f \in C$, $g \in C$, and $a,b \in \ZZ_p$
		\begin{center}
			$af+ bg \in C$;
		\end{center}
		
		\item[(2)] for all $s \in \NN$, $f \in C$, $l \leq \MInd(f)$, and $\mathbfsl{a} \in \mathbb{Z}^{ s}_p$:
		
		\begin{center}
			$g = f\circ_l (\sum_{i=1}^s(\mathbfsl{a})_ix_i)$ is in $C$.
		\end{center}
		
	\end{enumerate}
\end{Def}

Let $S \subseteq \mathbf{R}[X]$. Then we denote by \index{$\langle \rangle_{\mathbf{R}}$}$\langle S\rangle_{\mathbf{R}}$ the $\mathbf{R}$-\emph{polynomial linearly closed clonoid} generated by $S$. We can see that $\mathbf{R}^*[X]$ forms an $\mathbf{R}$-linearly closed clonoid.

Let us now modify \cite[Lemmata $3.8$ and $3.9$]{Kre.CFSO} to deal with $\mathbf{R}$-polynomial linearly closed clonoids. Since the statements of the following two lemmata are not the same of \cite[Lemmata $3.8$ and $3.9$]{Kre.CFSO} we show them with proofs that are essentially the same of \cite{Kre.CFSO}.

\begin{Lem}
	\label{Lempclonoids}
	Let $d \in \NN$, let $r \in \mathbf{R}$ and let $g \in\mathbf{R}[X]$ with $max(\TD(g)) \leq d$, and $\mathbf{1}_d \not\in \DEGS(g)$. Then $rx_1\cdots x_d\in \langle\{rx_1\cdots x_d + g\}\rangle_{\mathbf{R}}$.
\end{Lem}

\begin{proof}
	If $g = 0$ we have nothing to show. Let $g \not= 0$ and let $C := \langle\{rx_1\cdots x_d + g\}\rangle_{\mathbf{R}}$. We prove that there exists $g' \in \mathbf{R}[X]$ with $max(\TD(g')) \leq d$, $\mathbf{1}_d \not\in \DEGS(g')$ and $\Var(g') \subseteq \{x_1 ,\dots, x_d\}$ such that $rx_1\cdots x_d + g' \in C$. 
	
	Case $\Var(g) \subseteq \{x_1 ,\dots, x_d\}$: in this case the claim obviously holds.
	
	Case $\Var(g) \not\subseteq \{x_1 ,\dots, x_d\}$: let $B := \Var(g) \backslash \{x_1 ,\dots, x_d\}$. Let $k \in \NN$ and let $b_1 ,\dots, b_k \in \NN$ be such that $b_1 < \dots < b_k $ and $B = \{x_{b_1},\dots, x_{b_k}\}$. We set $g' \in \mathbf{R}[X]$ such that $g' := g \circ_{(b_1 ,...,b_k)} \ \mathbf{0}_k$. It follows that $\Var(g') \subseteq \{x_1,\dots,x_d\}$ and we can observe that:
	
	\begin{equation}
	rx_1\cdots x_d + g' = (rfx_1\cdots x_d + g) \circ_{(b_1 ,\dots,b_k)}\  \mathbf{0}_k \in C.
	\end{equation}
	Next we proceed by induction on the number of monomials of $g'$ in order to show that $rx_1\cdots x_d \in \langle \{rx_1\cdots x_d + g' \}\rangle_{\mathbf{R}} \subseteq C$. 
	
	If $g'$ has zero monomial then the claim obviously holds. Let us suppose that the claim holds for a number of monomials $t \geq 0$. Let the number of monomials of $g'$ be $t+1$. We observe that there exists $x_l \in \{x_1,\dots,x_d\}$ and a monomial $m$ of $g'$ such that $x_l$ does not appear in $m$. Thus we obtain:
	\begin{equation}
	rx_1\cdots x_d + g' - (rx_1\cdots x_d + g') \circ_{(l)} 0 = rx_1\cdots x_d + g'' \in C.
	\end{equation}
	Thus $g'' := g' - g' \circ_{(l)} 0$ satisfies the properties that $max(\TD(g'')) \leq d$, $\mathbf{1}_d \not\in \DEGS(g'')$, $\Var(g'') \subseteq \{x_1 ,\dots, x_d \}$ and $g''$ has fewer monomials than $g'$, since the monomial $m$ is cancelled in $g' - g' \circ_{(l)} 0$. By the induction hypothesis $rx_1\cdots x_d\in \langle\{rfx_1\cdots x_d + g''\}\rangle_{\mathbf{R}} \subseteq \langle\{rfx_1\cdots x_d + g'\}\rangle_{\mathbf{R}}$ and the claim holds.
\end{proof}

With Lemma \ref{Lempclonoids} we can now prove the following generalization of \cite[Lemma $3.9$]{Kre.CFSO} with the same proof.

\begin{Lem}
	\label{LemFoundPcloni}
	Let $d \in \NN$ and let $f$ be a polynomial in $\mathbf{R}^*[X]$ with $d := max(\TD(f))$. Let $m \in \MON(f)$ be a monomial with coefficient $r \in \mathbf{R}$ and $\tD(m) = d$. Then the following hold:
	\begin{enumerate}
		\item[(1)] if $d = 0$, then $r \in \langle\{f\}\rangle_{\mathbf{R}}$;
		\item[(2)] if $d > 0$, then $rx_1 \dots x_d \in \langle\{f\}\rangle_{\mathbf{R}}$.
	\end{enumerate}
	
\end{Lem}

\begin{proof}
	Let $n = \MInd(f)$, let $f = \sum_{\mathbfsl{m} \in [p-1]_0^n}r_{\mathbfsl{m}}\mathbfsl{x}^{\mathbfsl{m}} \in \mathbf{R}^*[X]$, $C := \langle\{f\}\rangle_{\mathbf{R}}$, let $d := max(\TD(f))$, and let $m \in \MON(f)$ with coefficient  $r \in \mathbf{R}$ and  $\tD(m) = d$. 
	
	Case $d = 0$: then clearly $r \in \langle\{f\}\rangle_{\mathbf{R}}$.
	
	Case $d > 0$: let us fix an $\mathbfsl{s} \in \DEGS(f)$ with $\langle \mathbfsl{s},\mathbf{1}_n\rangle = d$ and let $r$ be the coefficient of $\mathbfsl{x}^{\mathbfsl{s}}$ in $f$. Now we show by induction that for all $m \in \NN$ with $|\Var(\mathbfsl{x}^{\mathbfsl{s}})| \leq m \leq d$ there exists a monomial $h := r\mathbfsl{x}^{\mathbfsl{t}} \in \mathbf{R}^*[X]$ with $|\Var(h)| = m$, $\tD(h) = d$, and there exists $g \in \mathbf{R}^*[X]$ with $max(\TD(g)) \leq d$ and $\mathbfsl{t} \not\in \DEGS(g)$ such that $h + g \in C$.
	
	Base step $m = |\Var(\mathbfsl{x}^{\mathbfsl{s}})|$: in this case we see that we can take $h = r\mathbfsl{x}^{\mathbfsl{s}}$ and  $g= \sum_{\mathbfsl{m} \in [p-1]_0^n\backslash \{\mathbfsl{s} \}} r_{\mathbfsl{m}}\mathbfsl{x}^{\mathbfsl{m}}$.
	
	Induction step: suppose that the claim holds for $m < d$. Then there exists a monomial $h := r\mathbfsl{x}^{\mathbfsl{t}} \in \mathbf{R}^*[X]$ and $g' \in \mathbf{R}^*[X]$ with $|\Var(h)| = m$, $\tD(h) = d$, $max(\TD(g')) \leq d$, and $\mathbfsl{t} \not\in \DEGS(g')$ such that $h + g' \in C$. Then there exists $j$ such that $p-1\geq(\mathbfsl{t})_j = u > 1$, since $m < d$. Without loss of generality we suppose that $\Var(h) = \{x_1,\dots,x_m\}$. Let $g'' := g' \circ_{(j)} (x_j + x_{m+1})$. We denote by $\mathbfsl{y}$ the vector of variables $(x_1,\dots,x_{m+1})$. Thus 
	
	\begin{equation}
	\begin{split}
	(h+g') \circ_{(j)} (x_j + x_{m+1}) =& r\mathbfsl{x}^{\mathbfsl{t}}\circ(x_j + x_{m+1})+ g' \circ_{(j)} (x_j + x_{m+1})
	\\=&r( \sum_{k=0}^u {u \choose k}x^{u-k}_jx_{m+1}^k)  \cdot \prod_{i=1,i\not=j}^s(\mathbfsl{x})_i^{(\mathbfsl{t})_i} + g''
	\\=&r \cdot (\mathbfsl{t})_j \cdot \mathbfsl{y}^{((\mathbfsl{t})_1,\dots,(\mathbfsl{t})_{j-1},(\mathbfsl{t})_j -1,(\mathbfsl{t})_{j+1},\dots, (\mathbfsl{t})_m, 1)} +
	\\+&r( \sum_{k=0,k\not=1}^u{u\choose k}x^{u-k}_jx_{m+1}^k)\cdot \prod_{i=1,i\not=j}(\mathbfsl{x})_i^{(\mathbfsl{t})_i} + g''.
	\end{split}
	\end{equation}
	Let 
	
	\begin{equation}
	\begin{split}
	&h :=r \cdot \mathbfsl{y}^{((\mathbfsl{t})_1,\dots,(\mathbfsl{t})_{j-1},(\mathbfsl{t})_j -1,(\mathbfsl{t})_{j+1},\dots, (\mathbfsl{t})_m, 1)}
	\\&g= (\mathbfsl{t})_j^{-1}r( \sum_{k=0,k\not=1}^u{u\choose k}x^{u-k}_jx_{m+1}^k)\cdot \prod_{i=1,i\not=j}(\mathbfsl{x})_i^{(\mathbfsl{t})_i} + (\mathbfsl{t})_j^{-1}g''.
	\end{split}
	\end{equation}
	Then $h$ satisfies $\tD(h) = d$, $|\Var(h)| = m+1$. Furthermore, $g$ satisfies $max(\TD(g))$ $ \leq d$ and $((\mathbfsl{t})_1,\dots,(\mathbfsl{t})_{j-1},(\mathbfsl{t})_j -1,(\mathbfsl{t})_{j+1},\dots, (\mathbfsl{t})_m, 1) \not\in \DEGS(g)$. Thus $h$ and $g$ are the searched polynomials. This concludes the induction step.
	
	Thus, by Definition \ref{DefPolyClonoid}, $rx_1 \dots x_d + g''' \in C$ for some $g''' \in \mathbf{R}^*[X]$ with $max(\TD(g''') )\leq d$. By Lemma \ref{Lempclonoids} we have that $rx_1 \dots x_d \in C$ and the claim holds.
\end{proof}

We are now ready to prove that an $\mathbf{R}$-polynomial linearly closed clonoid generated by an element of $\mathbf{R}^*[X]$ contains every monomial of every polynomial in it.

\begin{Lem}
	\label{LemMoninZpq}
	Let $f \in \mathbf{R}^*[X]$ be such that $h = r_{\mathbfsl{m}}\mathbfsl{x}^{\mathbfsl{m}}$ is a monomial of $f$. Then $h \in \langle f \rangle_{\mathbf{R}}$.
\end{Lem}

\begin{proof}
	The proof is by induction on the number $n$ of monomials in $f$.
	
	Base step $n = 1$: then clearly the claim holds.
	
	Induction step $n>0$: suppose that the claim holds for every $g$ with $n-1$ monomials. Let $f$ be a polynomial with $n$ monomials with non-zero coefficients. Let $d = max(\TD(f))$ and let $h$ be a monomial in $f$ with degree $d$ and coefficient $r_h$. By Lemma \ref{LemFoundPcloni}, we have that either $r_{h}x_1\cdots x_d \in \langle f \rangle_{\mathbf{R}}$ if $d >0$ or $r_{h} \in \langle f \rangle_p$ if $d =0$. Clearly, this yields $h \in \langle f \rangle_{\mathbf{R}}$. From the induction hypothesis we have that all $n-1$ monomials of $f -h$ are in $\langle f -h \rangle_{\mathbf{R}} \subseteq \langle f  \rangle_{\mathbf{R}}$. Thus all monomials of $f$ are in $\langle f  \rangle_{\mathbf{R}}$.
\end{proof}

For every polynomial $f \in (\ZZ_p^{\ZZ_q^n})^*[X]$ with $\MInd(f) = k$ of the form $f =\sum_{\mathbfsl{m} \in [p-1]_0^k}$ $r_{\mathbfsl{m}}\mathbfsl{x}^{\mathbfsl{m}}$ we define its $s$-ary \emph{induced function} \index{$\overline{f}^{[s]}$}$\overline{f}^{[s]}: \ZZ_p^s \times \ZZ_q^s \rightarrow \ZZ_p \times \ZZ_q$ by:

\begin{center}
	$(\mathbfsl{x},\mathbfsl{y}) \mapsto (\sum_{\mathbfsl{m} \in [p-1]_0^k} r_{\mathbfsl{m}}((\mathbfsl{y})_1,\dots,(\mathbfsl{y})_n)\prod_{i =1}^k(\mathbfsl{x})_i^{(\mathbfsl{m})_i},0)$,
\end{center}
with $s \geq k,n$. From now on, when not specified, $s =  max(\{\MInd(f),n\})$. Next we show a lemma that connects the monomials of an $\mathbf{R}$-polynomial linearly closed clonoid to functions of a clones on $\ZZ_{pq}$.

\begin{Lem}
	\label{LemConnInduced}
	Let $p$ and $q$ be distinct prime numbers. Let $\mathbf{R} = \ZZ_p^{\ZZ_q^n}$ and let $h,h_1 \in \mathbf{R}^*[X]$ with $h \in \langle h_1 \rangle_{\mathbf{R}}$. Then $\overline{h} \in \Clg(\{\overline{h_1}\})$.
\end{Lem}

\begin{proof}
	We know that every clone $C$ containing $\Clo(\langle \ZZ_{pq},+\rangle)$ is closed under composition and, by Remark \ref{RemLinComb}, contains every linear mapping $h_{(\mathbfsl{a},\mathbfsl{b})}$ with $\mathbfsl{a} \in \ZZ_p^n$ and $\mathbfsl{b} \in \ZZ_q^n$. Then it is clear that if a monomial $h$ can be generated from $h_1$ with item $(1)$ or $(2)$ of Definition \ref{DefPolyClonoid}, then the induced monomial $\overline{h}_1$ generates $\overline{h}$ in a clone containing $\Clo(\langle \ZZ_{pq},+\rangle)$, simply composing $\overline{h}_1$ with the linear mappings $h_{(\mathbfsl{a},\mathbfsl{b})}$ from the right, with $\mathbfsl{a} \in \ZZ_p^n$ and $\mathbfsl{b} \in \ZZ_q^n$, and with the sum from the left.
\end{proof}

In order to show the following lemma we want to make an easy observation. Let $h = r\mathbfsl{x}^{\mathbf{1}_d} \in \mathbf{R}^*[X]$ be a monomial, where with \index{$\mathbfsl{x}^{\mathbf{1}_d}$}$\mathbfsl{x}^{\mathbf{1}_d}$ we denote the monomial $x_1\cdots x_{d}$ and let $d \in \NN\backslash\{1\}$. Furthermore, let $\mathbf{R} = \ZZ_p^{\ZZ_q^n}$. Relabelling the variables and composing $\overline{r\mathbfsl{x}^{\mathbf{1}_d}}$ with linear mappings and itself we can observe that $\Clg(\{\overline{r\mathbfsl{x}^{\mathbf{1}_d}}\})\supseteq \{\overline{r\mathbfsl{x}^{\mathbf{1}_d}}, \overline{r^2\mathbfsl{x}^{\mathbf{1}_{d + d-1}}}, \dots, \overline{r^k\mathbfsl{x}^{\mathbf{1}_{d + (k-1)(d-1)}}},\dots\}$. We have that $r^p = r$ modulo $p$, thus some of these monomials have the same coefficients. Moreover, if we substitute $x_1$ to $p$ variables we obtain that $p-1$ variables are cancelling, since $x_i^p$ and $x_i$ induce the same function and have the same representative in $\mathbf{R}^*[X]$. With these easy observations we can prove the following.

\begin{Lem}
	\label{LemMonVeri}
	Let $p$ and $q$ be distinct prime numbers. Let $d \in \NN\backslash\{1\}$. Then for all $k,l \in \NN$, for all $r \in \ZZ_p^{\ZZ_q^l}$, and for all $\mathbfsl{m} \in [p-1]_0^k\backslash\{\mathbf{0}_k\}$ with $\tD(\mathbfsl{x}^{\mathbfsl{m}})  = u$ congruent to $d$ modulo $p-1$ it follows that:
	\begin{center}
		$\overline{r\mathbfsl{x}^{\mathbfsl{m}}}\in \Clg(\{\overline{rx_1\cdots x_d}\})$.
	\end{center}

\end{Lem}

\begin{proof}
	
	Let $p$ be a prime number and let $d \in \NN\backslash\{1\}$. Clearly, $r\mathbfsl{x}^{\mathbfsl{m}} \in \langle rx_1\cdots x_{d'} \rangle_{\mathbf{R}}$, where $d' = \tD(\mathbfsl{x}^{\mathbfsl{m}})$, hence, by Lemma \ref{LemConnInduced}, we have that $\overline{r\mathbfsl{x}^{\mathbfsl{m}}}\in \Clg(\{\overline{rx_1\cdots x_{d'}}\})$.
	
	Next we prove that $\Clg(\{\overline{rx_1\cdots x_{d'}}\}) \subseteq \Clg(\{\overline{rx_1\cdots x_d}\})$. If $d' < d$ then, in order to generate $\overline{rx_1\cdots x_{d'}}$ we compose $\overline{rx_1\cdots x_d}$ with a linear mapping that is the first projection over $\ZZ_p$ and is the identity over $\ZZ_q$. In this way we collapse the variables $\{x_{d'+1},\dots,x_d\}$, since $x_1^{p} = x_1$ modulo $p$ and $d' = d$ modulo $p-1$.
	
	If $d' > d$, then we compose $\overline{rx_1\cdots x_d}$ to itself relabelling the variables and we generate the induced monomials $\{\overline{r\mathbfsl{x}^{\mathbf{1}_d}}, \overline{r^2\mathbfsl{x}^{\mathbf{1}_{d + d-1}}}, \dots, \overline{r^k\mathbfsl{x}^{\mathbf{1}_{d + (k-1)(d-1)}}},\dots\}$. We know that $r^p = r$ and thus in particular we generate $\{\overline{r\mathbfsl{x}^{\mathbf{1}_d}}, \overline{r\mathbfsl{x}^{\mathbf{1}_{d + (p-1)d}}}, \dots, $ $\overline{r\mathbfsl{x}^{\mathbf{1}_{d + c(p-1)d}}},\dots\}$ for all $c \in \NN$. For a suitable $k \in \NN$ we have that $d + k(p-1)d > d'$ and thus $\Clg(\{\overline{rx_1\cdots x_d}\}) \supseteq \Clg(\{\overline{rx_1\cdots x_{d + k(p-1)d}}\}) \supseteq \Clg(\{\overline{rx_1\cdots x_{d'}}\})$ and the claim holds.
	
\end{proof}

\begin{Lem}
	\label{Lemfcontmon}
	Let $p$ and $q$ be distinct prime numbers. Let $n \in \NN$, let $f: \ZZ_p^n \times \ZZ_q^n \rightarrow \ZZ_p \times \ZZ_q$ be an $n$-ary function, and let $g = q^{p-1}f$. Let $h \in \mathbf{R}^*[X]$ be such that $\mathbf{R} =\ZZ_p^{\ZZ_q^n}$ and $\overline{h} = g$. Let $h'$ be a monomial of $h$ with coefficient $r$ and $d = \tD(h')$. Then it follows that:
	
	\begin{enumerate}
		\item [(1)] if $d=0$, then $\overline{r}\in \Clg(\{f\})$;
		\item [(2)] if $d>0$, then $\overline{rx_1\cdots x_d}\in \Clg(\{f\})$.
	\end{enumerate}
\end{Lem}

\begin{proof}
	let $n$, $h$, and let $f$ be as in the hypothesis. By Lemma \ref{LemFoundPcloni}, we have that if $d=0$, then $r \in \langle h' \rangle_{\mathbf{R}}$ and if $d>0$, then $rx_1\cdots x_d \in \langle h' \rangle_{\mathbf{R}}$. By Lemma \ref{LemMoninZpq}, $h'\in \langle h \rangle_{\mathbf{R}}$ and thus, by Lemma \ref{LemConnInduced}, if $d=0$, then $\overline{r}\in \Clg(\{f\})$ holds, and if $d>0$, then $\overline{rx_1\cdots x_d}\in \Clg(\{f\})$.
\end{proof}

We are now ready to prove the main result of this section which allows us to provide a bound for the lattice of all clones containing $\Clo(\langle \ZZ_{pq},+\rangle)$.

\begin{proof}[Proof of Theorem \ref{Thmgeneralembedding}]
	Let $p$ and $q$ be distinct prime numbers. Then for all $1 \leq i \leq p$ and for all $1 \leq j \leq q$, we define $\rho_i: \Clo^{\mathcal{L}}(\langle \ZZ_{pq},+\rangle) \rightarrow \mathcal{L}(\ZZ_p,\ZZ_q)$ and $\psi_j: \Clo^{\mathcal{L}}(\langle \ZZ_{pq},+\rangle) \rightarrow \mathcal{L}(\ZZ_q,\ZZ_p)$ by:
	
	\begin{equation}
	\label{defembeGen}
	\begin{split}
	\rho_i(C) &:= \{f \mid \text{there exists } n \in \NN \text { s. t. }f:\ZZ_q^n \rightarrow \ZZ_p, \overline{fx_1\cdots x_i} \in C\}
	\\ \psi_j(C) &:= \{f \mid \text{there exists } n \in \NN \text { s. t. }f:\ZZ_p^n \rightarrow \ZZ_q, \overline{fx_1\cdots x_i} \in C\},
	\end{split}
	\end{equation}
	for all $C \in \Clo^{\mathcal{L}}(\langle \ZZ_{pq},+\rangle)$. Furthermore, we define  $\rho_0: \Clo^{\mathcal{L}}(\langle \ZZ_{pq},+\rangle) \rightarrow \mathcal{L}(\ZZ_p,\ZZ_q)$ and $\psi_0: \Clo^{\mathcal{L}}(\langle \ZZ_{pq},+\rangle) \rightarrow \mathcal{L}(\ZZ_q,\ZZ_p)$ by:
	
	\begin{equation}
	\label{defembeGen0}
	\begin{split}
	\rho_0(C) &:= \{f \mid \text{there exists } n \in \NN \text { s. t. }f:\ZZ_q^n \rightarrow \ZZ_p, \overline{f} \in C\}
	\\ \psi_0(C) &:= \{f \mid\text{there exists } n \in \NN \text { s. t. } f:\ZZ_p^n \rightarrow \ZZ_q, \overline{f} \in C\}.
	\end{split}
	\end{equation}
	Let $\rho: \mathcal{L}(\langle \ZZ_{pq},+\rangle)\rightarrow \mathcal{L}(\ZZ_p,\ZZ_q)^{p+1} \times \mathcal{L}(\ZZ_q,\ZZ_p)^{q+1}$ be defined by $\rho(C) = (\rho_0(C),\dots,$ $\rho_{p}(C),\psi_0(C),\dots,\psi_{q}(C))$ for all $C \in \Clo^{\mathcal{L}}(\langle \ZZ_{pq},+\rangle)$. 
	Next we prove that for all $0 \leq i \leq p$ and for all $0 \leq j \leq q$ $\rho_i$ and $\psi_j$ and thus $\rho$ are well-defined.
	
	Let $C \in \Clo^{\mathcal{L}}(\langle \ZZ_{pq},+\rangle)$. Then we have to prove that $\rho_i(C)$ is a $(\ZZ_p,\ZZ_q)$-linearly closed clonoid. To do so let $n \in \NN$, $f,g \in \rho_i(C)^{[n]}$ and $a, b \in \mathbb{Z}_p$. Then $\overline{fx_1\cdots x_i}, \overline{gx_1\cdots x_i}$ $ \in C$. From the closure with $+$ we have that $\overline{(af+bg)x_1\cdots x_i} \in C$ and thus $af + bg \in \rho_i(C)^{[n]}$ and item $(1)$ of Definition \ref{DefClo} holds. Furthermore, let $m,n \in \NN$, $f \in \rho_i(C)^{[m]}$, $A \in \mathbb{Z}^{m \times n}_q$ and let $g: \ZZ_{p}^n \rightarrow \ZZ_q$ be defined by:
	
	\begin{center}
		$g: (y_1,\dots,y_n) \mapsto f(A\cdot (y_1,\dots,y_n)^t)$.
	\end{center}
	Then, by definition of $\rho_i$, we have that $\overline{fx_1\cdots x_i} \in C^{[s]}$, where $s = max(\{i,m\})$. Let $k = max(\{i,n\})$ and let $h: \ZZ_p^k \times \ZZ_q^k \rightarrow  \ZZ_p^k \times \ZZ_q^k$ be a linear mapping such that $(x_1,\dots,x_k,y_1,\dots,y_k) \mapsto (
	x_1,\dots,x_k, A \cdot(y_1,\dots,y_n)^t,0_{\ZZ_q},\dots,0_{\ZZ_q})$. By definition of induced function we have that $\overline{gx_1\cdots x_i}= \overline{fx_1\cdots x_i} \circ h$ and hence $\overline{gx_1\cdots x_i} \in \Clg(\{$ $\overline{fx_1\cdots x_i}\})$ as composition of $\overline{fx_1\cdots x_i}$ and linear mappings. Thus $g \in \rho_i(C)$ which concludes the proof of item $(2)$ of Definition \ref{DefClo}. In the same way we can prove that $\rho_0$ is well-defined and $\psi_j$ is well-defined for all $0 \leq j \leq q$. Hence $\rho$ is well-defined. 
	
	Now we prove that $\rho$ is injective. Let $C, D \in \Clo^{\mathcal{L}}(\langle\ZZ_{pq},+\rangle)$ with $\rho(C) = \rho(D)$. Let $f \in C^{[n]}$ be such that $f$ satisfies for all $(\mathbfsl{x},\mathbfsl{y}) \in \ZZ_p^n \times \ZZ_q^n$:
	
	\begin{center}
		$f((\mathbfsl{x},\mathbfsl{y})) = (\sum_{\mathbfsl{m} \in [p-1]_0^n}f_{\mathbfsl{m}}\mathbfsl{x}^{\mathbfsl{m}}, \sum_{\mathbfsl{h} \in [q-1]_0^n}f_{\mathbfsl{h}}\mathbfsl{y}^{\mathbfsl{h}})$,
	\end{center}
	where $\{f_{\mathbfsl{m}}\}_{[p-1]_0^n}$ and $\{f_{\mathbfsl{h}}\}_{[q-1]_0^n}$ are sequences of functions respectively from $\ZZ_q^n$ to $\ZZ_p$ and from $\ZZ_p^n$ to $\ZZ_q$. Let $p_1 \in \mathbf{R}^*[X]$ be such that $\overline{p_1}^{[n]}  = q^{p-1}f$, where $\mathbf{R} = \ZZ_p^{\ZZ_q^n}$.

	Let $h = f_{\mathbfsl{m}}x^{\mathbfsl{m}}$ be a monomial of $p_1$ with $f_{\mathbfsl{m}} \not=0$, let $s = \tD(h)$. Let $d \in \NN_0$ be such that if $s \not= 0,1$, then $2 \leq d \leq p$ and $d=s$ modulo $p-1$. If $s = 0,1$, then $s = d$. We prove that $\overline{h} \in D$ by case distinction.
	
	Case $s=0,1$: from Lemma \ref{Lemfcontmon} it follows that $\overline{h} \in C$. By Definition (\ref{defembeGen0}), $f_{\mathbfsl{m}} \in \rho_{s}(C) = \rho_{s}(D)$ and thus $\overline{h} \in D$.
	
	Case $s>1$: by Lemma \ref{Lemfcontmon}, $C \supseteq \Clg(\{\overline{f_{\mathbfsl{m}}x_1\cdots x_s}\})$. Thus, by Lemma \ref{LemMonVeri}, $C \supseteq \Clg(\{\overline{f_{\mathbfsl{m}}x_1\cdots x_d}\})$ and thus $f_{\mathbfsl{m}} \in \rho_d(C) =  \rho_d(D)$. Hence $\overline{f_{\mathbfsl{m}}x_1\cdots x_d} \in D$ and, by Lemma \ref{LemMonVeri}, it follows that $\overline{f_{\mathbfsl{m}}\mathbfsl{x}^{\mathbfsl{m}}}  \in D$. This yields for a generic induced monomial in $q^{p-1}f$ and thus the function $q^{p-1}f \in D$. With the same strategy we can prove that $p^{q-1}f \in D$ and thus $f = p^{q-1}f + q^{p-1}f\in D$. Hence $C \subseteq D$. With the same proof we have the other inclusion and thus $\rho$ is injective.
	
\end{proof}

\begin{Cor}
	\label{Cor1}
	Let $p$ and $q$ be distinct prime numbers. Let $\prod_{i =1}^n p_i^{k_i}$ and $\prod_{i =1}^s r_i^{d_i}$ be the factorizations of respectively $g_1 = x^{q-1} -1$ in $\mathbb{Z}_p[x]$ and $g_2 = x^{p-1} -1$ in $\mathbb{Z}_q[x]$ into their irreducible divisors. Then the number $k$ of clones containing $\Clo(\langle\ZZ_{pq},+\rangle)$ is bounded by:
	
	\begin{equation}
	k \leq 2^{p+q+2}\prod_{i =1}^n(k_i +1)^{p+1}\prod_{i =1}^s(d_i +1)^{q+1}.
	\end{equation}
	
\end{Cor}

\begin{proof}
	The proof follows from the injective function of Theorem \ref{Thmgeneralembedding} and \cite[Theorem $1.3$]{Fio.CSOF}.
\end{proof}

We can see that in the worst case this bound is equal to $2^{2qp+q+p}$, when $g_1 = x^{q-1} -1$ in $\mathbb{Z}_p[x]$ and $g_2 = x^{p-1} -1$ in $\mathbb{Z}_q[x]$ can both be factorized with linear factors. There are many examples when this happens. Furthermore, a lower bound is given by the embedding of Theorem \ref{ThEmbClonoids}. 

\begin{proof}[Proof of Corollary \ref{Corfinale} ]
	
	The proof follows from Corollary \ref{Cor1} and Theorems \ref{ThEmbClonoids}, \cite[$1.3$]{Fio.CSOF}, observing the fact that we are counting two times the clone composed by all linear functions.
	
\end{proof}

With these two corollaries we have found a bound for the number of clones containing $\Clo(\langle\ZZ_{pq},+\rangle)$. With the next two results we can also find a concrete bound for the arity of the generators that we need to characterized these clones.

\begin{Thm}
	\label{Cor3}
	Let $p$ and $q$ be distinct prime numbers. Then a clone $C$ containing $\Clo(\langle\ZZ_{pq},+\rangle)$ is generated by the sets of functions:
	
	\begin{center}
		$L := \bigcup_{i=1}^p\{\overline{rx_1\cdots x_i} \mid r:\ZZ_q \rightarrow \ZZ_p, \overline{rx_1\cdots x_i} \in C\} \cup \{\overline{r} \mid r:\ZZ_q \rightarrow \ZZ_p, \overline{r} \in C\}$
		$R := \bigcup_{i=1}^q\{\overline{ry_1\cdots y_i} \mid r:\ZZ_p \rightarrow \ZZ_q, \overline{ry_1\cdots y_i} \in C\} \cup \{\overline{r} \mid r:\ZZ_p \rightarrow \ZZ_q, \overline{r} \in C\}$.
	\end{center}
	
\end{Thm}

\begin{proof}
	
	Let $C$ be a clone containing $\Clo(\langle\ZZ_{pq},+\rangle)$ and let $f \in C^{[n]}$ be such that $f$ satisfies for all $(\mathbfsl{x},\mathbfsl{y}) \in \ZZ_p^n \times \ZZ_q^n$:
	
	\begin{center}
		$f((\mathbfsl{x},\mathbfsl{y})) = (\sum_{\mathbfsl{m} \in [p-1]_0^n}f_{\mathbfsl{m}}\mathbfsl{x}^{\mathbfsl{m}}, \sum_{\mathbfsl{h} \in [q-1]_0^n}f_{\mathbfsl{h}}\mathbfsl{y}^{\mathbfsl{h}})$,
	\end{center}
	where $\{f_{\mathbfsl{m}}\}_{[p-1]_0^n}$ and $\{f_{\mathbfsl{h}}\}_{[q-1]_0^n}$ are sequences of functions respectively from $\ZZ_q^n$ to $\ZZ_p$ and from $\ZZ_p^n$ to $\ZZ_q$. Let $p_1 \in \mathbf{R}^*[X]$ be such that $\overline{p_1}^{[n]}  = q^{p-1}f$, where $\mathbf{R} = \ZZ_p^{\ZZ_q^n}$. 
	
	Let $h = f_{\mathbfsl{m}}x^{\mathbfsl{m}}$ be a monomial of $p_1$ and let $s = tD(h)$. Then, by Lemmata \ref{LemMoninZpq} and \ref{LemConnInduced}, we have that $\overline{h} \in C$. Furthermore, let $d \in \NN_0$ be such that if $s \not= 0,1$, then $2 \leq d \leq p$ and $d=s$ modulo $p-1$. If $s=0,1$ then $s = d$. Then, by Lemmata \ref{LemMonVeri} and \ref{Lemfcontmon} it follows that if $s= 0$ $\Clg(\overline{h}) = \Clg(\{\overline{f_{\mathbfsl{m}}}\})$ and $\Clg(\overline{h}) = \Clg(\{\overline{f_{\mathbfsl{m}}x_1\cdots x_s}\}) = \Clg(\{\overline{f_{\mathbfsl{m}}x_1\cdots x_d}\}) $ otherwise. Then let us consider the $(\ZZ_p,\ZZ_q)$-linearly closed clonoid generated by $f_{\mathbfsl{m}}$. By Theorem \cite[Theorem $1.2$]{Fio.CSOF}, there exists a unary function $f:\ZZ_q \rightarrow \ZZ_p$ such that $\Cid(\{f\}) = \Cid(\{f_{\mathbfsl{m}}\})$. Hence, by the embedding of Theorem \ref{Thmgeneralembedding}, we have that $\Clg(\{\overline{f_{\mathbfsl{m}}}\}) = \Clg(\{\overline{f}\})$ and $\Clg(\{\overline{f_{\mathbfsl{m}}x_1\cdots x_i}\}) =\Clg(\{\overline{fx_1\cdots x_i}\})$ for all $i \in [p]$. Hence $\overline{h} \in \Clg(L)$ and thus $q^{p-1}f \in \Clg(L)$ since $\Clg(L)$ contains every induced monomial in $q^{p-1}f$.
	
	In the same way we can observe that $p^{q-1}f \in \Clg(R)$ and thus $f = q^{p-1}f + p^{q-1}f \in \Clg(L) \cup \Clg(R)$ and the claim holds.
	
\end{proof}

The proof of Corollary \ref{CorArFun} follows directly from Theorem \ref{Cor3} and gives an important connection between a clone $C$ containing $\Clo(\langle\ZZ_{pq},+\rangle)$ and its subsets of generators $L$ and $R$, when $p$ and $q$ are distinct primes. Indeed, Theorem \ref{Cor3} provides a complete list of generators for a clone containing $\Clo(\langle\ZZ_{pq},+\rangle)$ that is often redundant but explains how deep is the link between clonoids and clones in this case. We can observe that the generators in Theorem \ref{Cor3} are formed by a product of a unary function generating a $(\ZZ_p,\ZZ_q)$-linearly closed clonoid and a monomial generating a clone on $\ZZ_p$. This justifies the use of polynomials to represent functions of a clone containing $\Clo(\langle\ZZ_{pq},+\rangle)$ which gives a different prospective to these functions.

\section{Clones containing $\Clo(\langle\ZZ_{pq},+\rangle)$ which preserve $\pi_1$ and $[\pi_1,\pi_1] =0$}

In this section our aim is to characterize clones containing $\Clo(\langle\ZZ_{pq},+\rangle)$ which preserve $\pi_1$ and $[\pi_1,\pi_1] =0$. We will show there is an injective function from the lattice of all clones containing $\Clo(\langle \ZZ_{pq},+\rangle)$ which preserve $\pi_1$ and $[\pi_1,\pi_1] \leq 0$ to direct product of the lattice of all clones containing $\Clo(\langle \ZZ_{p},+\rangle)$ and the square of the lattice of all $(\ZZ_q,\ZZ_p)$-linearly closed clonoids. 

Let us start showing a general expression of a function in a clone containing $\Clo(\langle\ZZ_{pq},+\rangle)$ which preserves $\pi_1$ and $[\pi_1,\pi_1] =0$.

\begin{Lem}\label{Lem3}
	Let $p$ and $q$ be distinct prime numbers and let $f: \ZZ_p^n \times \ZZ_q^n \rightarrow \ZZ_p \times \ZZ_q$ be an $n$-ary function which preserves $\pi_1$. Then $f$ preserves $[\pi_1, \pi_1] = 0$ if and only if there exist $f_c: \ZZ_p^n \rightarrow \ZZ_q$, $\{a_{\mathbfsl{m}}\}_{\mathbfsl{m} \in [p-1]_0^n} \subseteq  \ZZ_p$ and $\mathbfsl{f}_1: \ZZ_p^n\rightarrow \ZZ_q^n$ such that for all $(\mathbfsl{x} ,\mathbfsl{y}) \in \ZZ_p^n \times \ZZ_q^n$, $f$ satisfies:
	
	\begin{equation}\label{prop1}
	f(\mathbfsl{x}, \mathbfsl{y}) = (\sum_{\mathbfsl{m} \in [p-1]_0^n}a_{\mathbfsl{m}}\mathbfsl{x}^\mathbfsl{m}, \langle \mathbfsl{f}_1(\mathbfsl{x}),\mathbfsl{y}\rangle + f_c(\mathbfsl{x})).
	\end{equation}
	
\end{Lem}

\begin{proof}
	
	Let us prove $\Rightarrow$. Let $f: \ZZ_p^n \times \ZZ_q^n \rightarrow \ZZ_p \times \ZZ_q$ be an $n$-ary function which preserves $\pi_1$ and $[\pi_1,\pi_1] = 0$. By Lemma \ref{Lem1-3} item $(1)$, we have that there exist $f_p: \ZZ^n_p \rightarrow \ZZ_p$ and  $f_q: \ZZ^n_p \times \ZZ_q^n \rightarrow \ZZ_q$ such that $f$ satisfies $f(\mathbfsl{x},\mathbfsl{y}) = (f_p(\mathbfsl{x}),f_q(\mathbfsl{x},\mathbfsl{y}))$ for all $(\mathbfsl{x},\mathbfsl{y}) \in \ZZ_p^n \times \ZZ_q^n$. Moreover, by item $(3)$ of Lemma \ref{Lem1-3} we have that $f$ is affine in the second component. Thus let us fix $\mathbfsl{a} \in \ZZ_p^n$. By Remark \ref{RemAffFun} for all $\mathbfsl{y} \in \ZZ_q^n$:
	
	\begin{equation}
	\label{eq:}
	f_q(\mathbfsl{a},\mathbfsl{y}) = \langle\mathbfsl{b}_{\mathbfsl{a}},\mathbfsl{y}\rangle + c_{\mathbfsl{a}},
	\end{equation}
	for some $\mathbfsl{b}_{\mathbfsl{a}} \in \ZZ_q^n$. Hence, using the Lagrange interpolation functions (\cite[Definition $4.2$]{Fio.CSOF}), we have that:
	
	\begin{equation}
	\label{eq2:}
	f_q(\mathbfsl{x},\mathbfsl{y}) = \sum_{\mathbfsl{a} \in \ZZ_p^n}f_{\mathbfsl{a}}(\mathbfsl{x})(\langle\mathbfsl{b}_{\mathbfsl{a}},\mathbfsl{y}\rangle + c_{\mathbfsl{a}}),
	\end{equation}
	for all $\mathbfsl{y} \in \ZZ_q^n$. Thus $f$ satisfies (\ref{prop1}) with $\mathbfsl{f}_1 = \sum_{\mathbf{a} \in \ZZ_p^n}f_{\mathbfsl{a}}\mathbfsl{b}_{\mathbfsl{a}}$ and $f_c = \sum_{\mathbf{a} \in \ZZ_p^n}f_{\mathbfsl{a}}c_{\mathbfsl{a}}$. Moreover, by Remark \ref{RemPolCom}, there exists a sequence $\{a_{\mathbfsl{m}}\}_{\mathbfsl{m} \in \ZZ_p^n}$ such that $f_1(\mathbfsl{x}) = \sum_{\mathbfsl{m} \in [p-1]_0^n}a_{\mathbfsl{m}}\mathbfsl{x}^\mathbfsl{m}$ for all $\mathbfsl{x} \in \ZZ_p^n$. Hence the $\Rightarrow$ implication holds. For the $\Leftarrow$ implication suppose that there exist $f_c: \ZZ_p^n \rightarrow \ZZ_q$, $\{a_{\mathbfsl{m}}\}_{\mathbfsl{m} \subseteq [p-1]_0^n} \subseteq  \ZZ_p$ and $\mathbfsl{f}_1: \ZZ_p^n\rightarrow \ZZ_q^n$ such that for all $(\mathbfsl{x} ,\mathbfsl{y}) \in \ZZ_p^n \times \ZZ_q^n$, $f$ satisfies (\ref{prop1}). Let $\mathbfsl{x}_1, \mathbfsl{x}_2, \mathbfsl{x}_3, \mathbfsl{x}_4 \in \ZZ_p^n$ and $\mathbfsl{y}_1, \mathbfsl{y}_2, \mathbfsl{y}_3, \mathbfsl{y}_4 \in \ZZ_q^n$ such that $\mathbfsl{x}_1 = \mathbfsl{x}_2 = \mathbfsl{x}_3 = \mathbfsl{x}_4$ and $\mathbfsl{y}_1 - \mathbfsl{y}_2 + \mathbfsl{y}_3 = \mathbfsl{y}_4$. Then:
	
	\begin{equation}
	\label{eqk:}
	\begin{split}
	f(\mathbfsl{x}_1,\mathbfsl{y}_1) - f(\mathbfsl{x}_2,\mathbfsl{y}_2) + f(\mathbfsl{x}_3,\mathbfsl{y}_3)  &= (\sum_{\mathbfsl{m} \in [p-1]_0^n}a_{\mathbfsl{m}}\mathbfsl{x}_1^\mathbfsl{m},\langle \mathbfsl{f}_1(\mathbfsl{x}_1),\mathbfsl{y}_1 - \mathbfsl{y}_2 + \mathbfsl{y}_3\rangle + f_c(\mathbfsl{x}_1)) \\&= f(\mathbfsl{x}_4,\mathbfsl{y}_4)
	\end{split}
	\end{equation}	
	and thus $f$ preserves $[\pi_1,\pi_1] = 0$.
	
\end{proof}

Let us now provide a proof of Theorem \ref{Thmembaddingg}. Let us consider the functions $\rho_1: \mathbf{\Clo}^{\diamond}(\langle\ZZ_{pq},+\rangle) $ $\rightarrow  \mathbf{\Clo}^{\mathcal{L}}(\langle\ZZ_{p},+\rangle)$ as defined in the proof of  Theorem \ref{ThmembeddingClones}, $\psi_0,\psi_1: \Clo^{\mathcal{L}}(\langle \ZZ_{pq},+\rangle) \rightarrow \mathcal{L}(\ZZ_p,\ZZ_q)$ as defined in the proof of Theorem \ref{Thmgeneralembedding}.

\begin{proof}[Proof of Theorem \ref{Thmembaddingg}]

	Let $p$ and $q$ be distinct prime numbers and let $\mathbf{\Clo}'(\langle\ZZ_{pq},+\rangle)$ be the lattice of all clones containing $\Clo(\langle\ZZ_{pq},+\rangle)$ which preserve $\pi_1$ and $[\pi_1,\pi_1] \leq 0$. Let $\psi: \Clo'(\langle\ZZ_{pq},+\rangle) \rightarrow \Clo^{\mathcal{L}}(\langle\ZZ_{p},+\rangle) \times \mathcal{L}(\ZZ_p,\ZZ_q)^2$ be defined by $\psi(C) := (\rho_1(C),\psi_1(C),\psi_0(C))$ for all $C \in \Clo'(\langle\ZZ_{pq},+\rangle)$. With the same proof of Theorems \ref{ThmembeddingClones} and \ref{Thmembaddingg} we see that $\rho_1,\psi_0,\psi_1$ and thus $\psi$ are well-defined.
	
	Next we prove that $\psi$ is injective. To this end let $C,D \in \Clo'(\langle\ZZ_{pq},+ \rangle)$ be such that $\psi(C) = \psi(D)$. Let us suppose that $f \in C^{[n]}$. By Lemma \ref{Lem3}, there exist $f_c: \ZZ_p^n \rightarrow \ZZ_q$, $\{a_{\mathbfsl{m}}\}_{\mathbfsl{m} \in  [p-1]_0^n} \subseteq  \ZZ_p$ and $\mathbfsl{f}_1: \ZZ_p^n\rightarrow \ZZ_q^n$ such that for all $(\mathbfsl{x} ,\mathbfsl{y}) \in \ZZ_p^n \times \ZZ_q^n$, $f$ satisfies:
	
	\begin{equation}\label{expr}
	f(\mathbfsl{x}, \mathbfsl{y}) = (\sum_{\mathbfsl{m} \in [p-1]_0^n}a_{\mathbfsl{m}}\mathbfsl{x}^\mathbfsl{m}, \langle \mathbfsl{f}_1(\mathbfsl{x}),\mathbfsl{y}\rangle + f_c(\mathbfsl{x})).
	\end{equation}
	It is clear that the functions $\{f_i: \ZZ_p^n \times \ZZ_q^n \rightarrow \ZZ_p \times \ZZ_q\}_{1 \leq i \leq n}$, $f'_c: \ZZ_p^n \times \ZZ_q^n \rightarrow \ZZ_p \times \ZZ_q$, and $t: \ZZ_p^n \times \ZZ_q^n \rightarrow \ZZ_p \times \ZZ_q$ such that for all $(\mathbfsl{x},\mathbfsl{y}) \in \ZZ_p^n \times \ZZ_q^n$, $f_i(\mathbfsl{x},\mathbfsl{y}) = (0, (\mathbfsl{f}_1(\mathbfsl{x}))_i(\mathbfsl{y})_i)$, $f'_c(\mathbfsl{x},\mathbfsl{y}) = (0, f_c(\mathbfsl{x}))$, and $t(\mathbfsl{x},\mathbfsl{y}) = (\sum_{\mathbfsl{m} \in [p-1]_0^n}a_{\mathbfsl{m}}\mathbfsl{x}^\mathbfsl{m}, 0)$ are in $C$. Thus $\pi_p^{\ZZ_p \times \ZZ_q} \circ t|{\ZZ_p^n \times\{\mathbf{0}\}} \in \rho_1(C) = \rho_1(D)$, $(\mathbfsl{f}_1)_i \in \psi_1(C) = \psi_1(D)$ for all $i \in [n]$, and $f_c \in \psi_0(C) = \psi_0(D)$. Hence, by definition of $\rho,\psi_0,\psi_1$, this implies that $p,f_1,\dots,f_n,f'_c \in D$. Hence $f = t + \sum_{i=1}^n f_i +f'_c \in D$. Thus $C \subseteq D$ and in the same way we can prove the other inequality. Hence $C = D$ and the claim holds.

\end{proof}

\begin{Cor}
	\label{Corpi1case}
	Let $p$ and $q$ be distinct prime numbers. Let $\prod_{i =1}^n p_i^{k_i}$ be the factorization of the polynomial $g = x^{p-1} -1$ in $\mathbb{Z}_q[x]$ into its irreducible divisors. Then the number of clones containing $\Clo(\langle\ZZ_{pq},+\rangle)$ which preserve $\pi_1$ and $[\pi_1,\pi_1] \leq 0$ is bounded from above by $(n(p)+3)(2\prod_{i =1}^n(k_i +1))^2$, where $n(k)$ is the number of divisors of $k-1$.
\end{Cor}

\begin{proof}
	We know from \cite[Corollary $1.3$]{Kre.CFSO} and \cite{Pos.TTVI} that number of clones containing $\Clo(\langle\ZZ_{p},+\rangle)$ is $n(p)+3$ where $n(p)$ is the number of divisors of $p-1$. Furthermore, we know from \cite[Theorem $1.3$]{Fio.CSOF} that the cardinality of the lattice of all $(\ZZ_q,\ZZ_p)$-linearly closed clonoids is $2\prod_{i =1}^n(k_i +1)$, where $\{k_i\}_{i \in [n]}$ are the exponents of the factorization of the polynomial $g = x^{p-1} -1$ in $\mathbb{Z}_q[x]$. Then the bound directly follows from the injective function of Theorem \ref{Thmembaddingg}. 
\end{proof}

We can observe that the bound of Corollary \ref{Corpi1case} is not reached. In particular the images of $\psi_0$ and $\psi_1$, are $(\ZZ_q,\ZZ_p)$-linearly closed clonoids that satisfy some further closure properties. Indeed, both are closed from the right not only under composition with linear mappings but also with functions of the clone image of $\rho_1$. Furthermore, the image of $\psi_1$ contains constants and is closed under point-wise product. The image of $\psi_0$ is closed under point-wise product with functions in $C_2$. These structures are related to vector subspaces of $\ZZ_p^{q}$ that are closed under Hadamard product. We will not go into details in the case, nevertheless we think that they are interesting structures.

In general we cannot conclude that $\rho$ is an embedding since the join of two $(\ZZ_q,\ZZ_p)$-linearly closed clonoids that satisfy the previous properties could also not satisfy them.

\section*{Acknowledgements}

The author thanks Erhard Aichinger, who inspired this paper, and Sebastian Kreinecker for many hours of fruitful discussions.

\bibliographystyle{alpha}

\end{document}